\documentclass[12pt,a4paper,twoside]{amsart}
\usepackage{amsmath,amssymb,amsfonts,amsthm,amsopn,dsfont}
\usepackage[dvips]{graphicx}
\usepackage{setspace}
\newcommand{\tfr}{time-frequency representation}

\newcommand{\stft}{short-time Fourier transform}

\newcommand{\Wd}{Wigner distribution}

\newcommand{\modsp}{modulation space}

\newtheorem{theorem}{Theorem}[section]
\newtheorem{corollary}[theorem]{Corollary}

\newtheorem{lemma}[theorem]{Lemma}
\numberwithin{equation}{section}

\newtheorem{proposition}[theorem]{Proposition}%%
\newtheorem{remark}[theorem]{Remark}

\newcommand{\beqa}{\begin{eqnarray*}}
\newcommand{\eeqa}{\end{eqnarray*}}

\newcommand{\field}[1]{\mathbb{#1}}
\newcommand{\bR}{\field{R}}        %  real numbers
\newcommand{\bZ}{\field{Z}}        %  whole numbers
\def\la{\lambda}

 \def\cF{\mathcal{F}}              % Calligraphic Letters
 \def\cS{\mathcal{S}}

 \def\cU{\mathcal{U}}

 \def\cC{\mathcal{C}}

\def\a{\aleph}

\def\vgf{V_gf}

\def\rd{\bR^d}

\def\rdd{{\bR^{2d}}}

\def\lrdd{L^2(\rdd)}
\def\zd{\bZ^d}

\def\intrd{\int_{\rd}}
\def\intrdd{\int_{\rdd}}

\def\R{\right)}

\def\<{\left<}
\def\>{\right>}

\def\mv1{M_v^1}

\def\mpq{M^{p,q}}

\def\phas{(x,\o )}
\def\mn{(m,n)}
\def\mn'{(m',n')}

\def\o{\omega}
\def\a{\alpha}
\def\b{\beta}
\def\z{\zeta}

\def\R{\mathbb{R}}
\def\Ren{\mathbb{R}^d}
\def\Renn{\mathbb{R}^{2d}}

\def\sch{\mathcal{S}}

\def\Fur{\mathcal{F}}

\def\f{\varphi}

\def\Sn2{S_{2}(L^{2}(\Ren))}
\def\S1{S_{1}(L^{2}(\Ren))}
\def\sig00{\sigma_{0,0}}
\def\la{\langle}
\def\ra{\rangle}

\begin{document}
%=================================================================
%=================================================================
%=================================================================
%=================================================================

\begin{abstract}
We give a complete
characterization of the
continuity of
pseudodifferential operators
with symbols in modulation
spaces $M^{p,q}$, acting on a
given Lebesgue space $L^r$.
Namely, we find the full
range of triples $(p,q,r)$,
for which such a boundedness
occurs. More generally, we
completely characterize the
same problem for operators
acting on Wiener amalgam
space $W(L^r,L^s)$ and even
on modulation spaces
$M^{r,s}$. Finally the action
of pseudodifferential
operators with symbols in
$W(\Fur L^1,L^\infty)$ is
also investigated.
\end{abstract}

\title[Pseudodifferential operators on Lebesgue spaces]{
 Pseudodifferential operators on $L^p$, Wiener
amalgam and modulation spaces}
\author{Elena Cordero and Fabio Nicola}
\address{Department of Mathematics,  University of Torino,
Via Carlo Alberto 10, 10123
Torino, Italy}
\address{Dipartimento di Matematica, Politecnico di
Torino, Corso Duca degli
Abruzzi 24, 10129 Torino,
Italy}
\thanks{The second author was partially supported by
the Progetto MIUR
Cofinanziato 2007 ``Analisi
Armonica''}
%\curraddr{}
\email{elena.cordero@unito.it}
\email{fabio.nicola@polito.it}
\date{}

\maketitle
%\footnotetext[0]{This paper is published despite the effects of the Italian law 133/08. This law drastically reduces public funds to public Italian universities, which is particularly dangerous for scientific free research, and it will prevent young researchers from getting a position, either temporary or tenured, in Italy. The authors are protesting against this law to obtain its cancellation. See http://groups.google.it/group/scienceaction.}

\section{Introduction}
A pseudodifferential operator
in $\R^d$ with symbol
$a\in\cS'(\R^d)$ is defined
by the formula
\begin{equation}\label{KN}
a(x,D)f(x)=\intrd a\phas \hat f(\o)
e^{2\pi i x\o}\,d\o,\qquad
f\in\cS(\R^d),
\end{equation}
where
$\hat{f}(x)=\Fur{f}(x)=\int_{\R^d}
e^{-2\pi i x\omega} f(x)\,dx$ is the
Fourier transform of $f$. Hence
$a(x,D)f$ is well-defined as a
temperate distribution.\par
Pseudodifferential operators arise at
least in three different frameworks:
partial differential equations (PDEs),
quantum mechanics and engineering.  In
PDEs they were introduced independently
in \cite{hormander65} and
\cite{kohn-nirenberg65}. Since then,
many symbol classes have been
considered, according to several
applications to PDEs. In particular, a
deep analysis of such operators have
been carried on for H\"ormander's
classes $S^m_{\rho,\delta}$, $m\in\R$,
$0\leq\delta\leq\rho\leq1$, of smooth
functions $a(x,\omega)$ satisfying the
estimates
$|\partial^\alpha_x\partial^\beta_\omega
a(x,\omega)|\leq
C_{\alpha,\beta}(1+|\omega|)^{m+\delta|\alpha|-\rho|\beta|}$.

Bundedness results on $L^p$-based
Sobolev spaces for those operators are
of special interest because they imply
regularity results for the solutions of
the corresponding PDEs.\par The basic
result in this connection is the
boundedness on $L^2$ of operators in
the above classes, with $\delta<\rho$,
which can be achieved by means of the
symbolic calculus. Indeed,
$L^2$-boundedness still holds for
$0\leq\delta=\rho<1$ and even for
symbols in $\cC^{2d+1}(\R^{2d})$, which
is the classical
Calder\'on-Vaillancourt theorem
\cite{calderon71,calderon72}.
Boundedness on $L^p$, $1<p<\infty$,
holds for symbols in $S^0_{1,\delta}$,
$0\leq\delta<1$, but generally fails
for $\rho<1$ and a loss of derivatives
may then occur. We refer the reader to
\cite{hormander-book,stein93,taylor1}
and the references therein for a
detailed account. There are also many
results for symbols which are smooth
and behaves as usual with respect to
$\o$, but less regular with respect to
$x$, e.g. just belonging to some
H\"older class. In this connection see
the books \cite{taylor1,taylor2}, where
important applications to nonlinear
equations are presented as well.\par
The smoothness of the symbol or the
boundedness of all derivatives of the
symbol are not necessary for the
boundedness of pseudodifferential
operators on $L^2(\rd)$. Being
motivated by this argument, many
authors (see, e.g.,
\cite{CM78,Co75,ka76,Na77}) contributed
to investigate the minimal assumption
on the regularity of symbols for the
corresponding operators to be bounded
on $L^2$. In particular, Sugimoto
\cite{Su88} showed that symbols in the
Besov space
$B_{d/2,d/2}^{(\infty,\infty),(1,1)}$
imply $L^2$-boundedness (see also
\cite{sugimoto2} and the references
therein for extensions to the $L^p$
framework). In 1994/95 Sj\"ostrand
introduced a new symbol class, larger
than $S^0_{0,0}$, which was then
recognized to be  the modulation space
$M^{\infty,1}(\R^d)$, first introduced
in time-frequency analysis by
Feichtinger
\cite{feichtinger80,F1,feichtinger83}.
For $1\leq p,q\leq\infty$, the
modulation space $M^{p,q}(\R^d)$
consists of the temperate distributions
$f$ such that the function
$\Fur(\overline{g(\cdot-x)}f)(\omega)$
belongs to the mixed-norm space
$L^{p,q}$ (see \eqref{normamista}
below), where $g$ (the so-called
window) is any non-zero Schwartz
function. The role of the factor
$g(\cdot-x)$ is that of localizing $f$
near the point $x$. Roughly speaking,
distributions in $M^{p,q}$ have
therefore the same local regularity as
a function whose Fourier transform is
in $L^q$, but decay at infinity like a
function in $L^p$ (see \cite{book} and
Section \ref{section2} below for
details). In
\cite{sjostrand94,sjostrand95}
Sj\"ostrand proved that symbols in
$M^{\infty,1}$ give rise to
$L^2$-bounded operators. In view of the
inclusion $\cC^{2d+1}(\R^{2d})\subset
M^{\infty,1}(\R^d)$, this result
represented an important generalization
of the classical Calderon-Vaillancourt
Theorem. Since then, several extensions
appeared, mostly due to Gr\"ochenig and
collaborators. In particular, in
\cite{book,GH99}, symbols in
$M^{\infty,1}$ were proved to produce
bounded operators on all $M^{p,q}$.
Further refinements appeared in
\cite{grochenig-ibero,labate2,Toft04,Toftweight}.\par
We now come more specifically to the
results of the present paper. Examples
show that symbols merely on $L^\infty$
generally do not produce bounded
operator in $L^2$, but some additional
regularity condition should be assumed.
The above Sj\"ostrand's result is just
an instance of this. There is a
 space larger than $M^{\infty,1}$ which
still consists of bounded functions
having locally the same regularity as a
function whose Fourier transform is
integrable. It is  the so-called Wiener
amalgam space $W(\Fur L^1,L^\infty)$,
the sub-space  of temperate
distributions $f$ such that the
function
$\Fur(\overline{g(\cdot-x)}f)(\omega)$
belongs to $L^\infty_x L^1_\omega$ (see
\eqref{normainversa} below). A natural
question which arises is whether
pseudodifferential operators with
symbols  in $W(\Fur L^1,L^\infty)$ are
$L^2$-bounded. Fourier multipliers with
symbols in $W(\Fur L^1,L^\infty)$ are
indeed bounded on $L^2$ and the same
holds, more generally, for symbols in
$W(\Fur L^1,L^\infty)$ of the type
$a(x,\omega)=a_1(x)a_2(\omega)$ (see
Proposition \ref{tensoriale} below).
However, contrary to what these special
cases could suggest, we shall show in
Proposition \ref{t27bis2} that, for
more general symbols in that class,
boundedness on $L^2$ may fail.
\par Another natural question
is which modulation spaces
give rise to bounded
operators on $L^p$,
$p\not=2$. We do not know
results in this connection in
the existent literature. We
give here a complete answer
to this problem, in the
following form (see Corollary
\ref{29-1}, Proposition \ref
{t41bis-1} and Figure
1).\par\medskip {\it Let
$1\leq p,q,r\leq\infty$ such
that
\begin{equation}\label{31-0}
\frac{1}{p}\geq
\left|\frac{1}{r}-\frac{1}{2}\right|+\frac{1}{q'},\qquad
q\leq\min\{r,r'\}.
\end{equation}
Then every symbol $a\in M^{p,q}$ gives
rise to a bounded operator $a(x,D)$ on
$L^r$.\\ Viceversa, if this conclusion
holds true, then the constraints in
\eqref{31-0} must be satisfied.
}\par\medskip\noindent
\begin{center}
\includegraphics[width=7truecm]{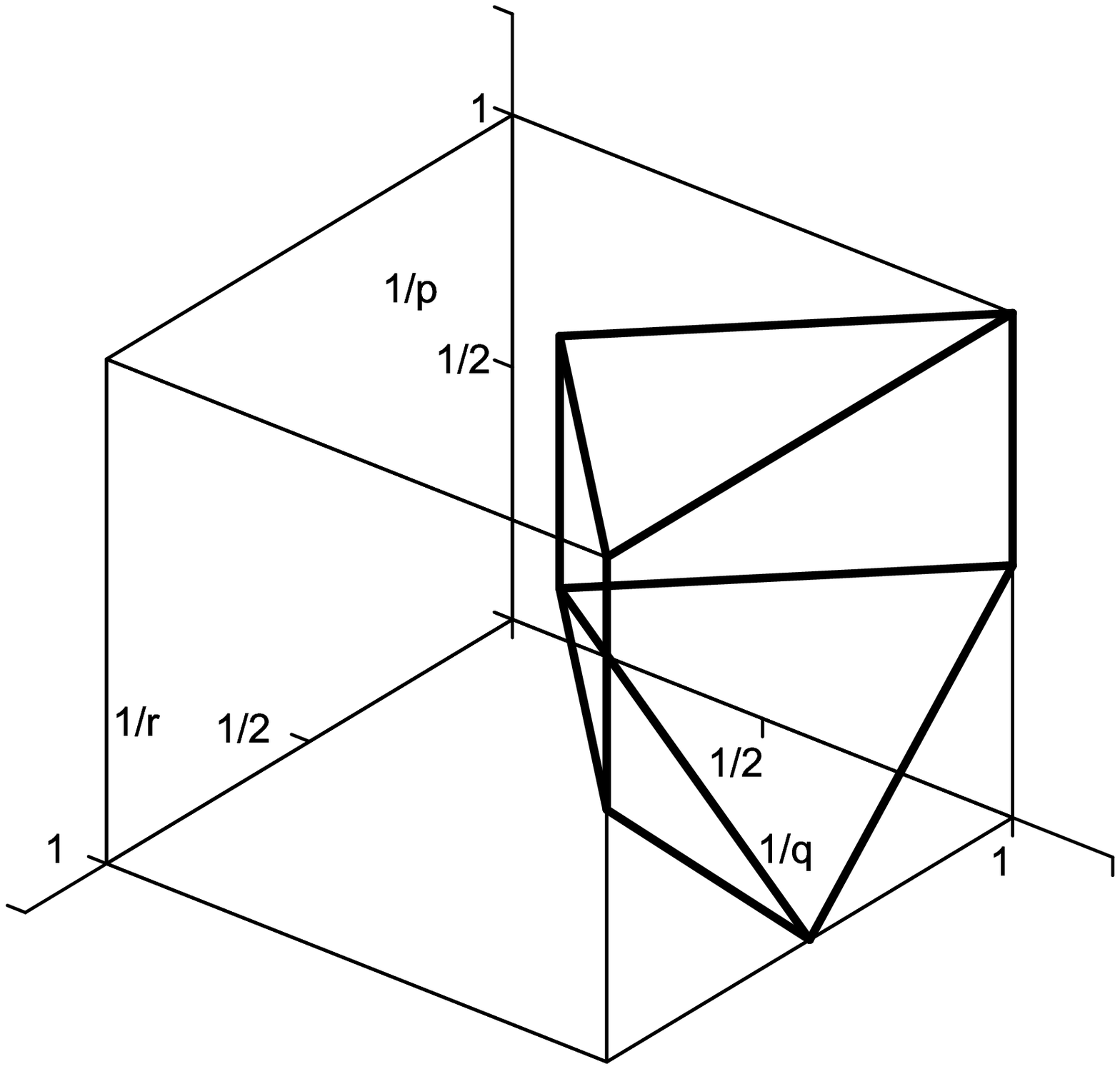}
\end{center}

           \begin{center}{Figure 1:
           \small The triples $(1/r,1/q,1/p)$ inside
            the convex polyhedron are exactly those for which every symbol in $M^{p,q}$ produces a bounded operator on $L^r$.}
           \end{center}

 To avoid
technicalities, we only consider the
action of $a(x,D)$ on Schwartz
functions, so that the definition of
boundedness which is relevant here
requires a small subtlety when
$r=\infty$; see Section
\ref{section4}.\par
 Actually,
we address to the more general problem
of boundedness on the so-called Wiener
amalgam spaces $W(L^p,L^q)$, $1\leq
p,q\leq\infty$, which generalize the
Lebesgue spaces. We recall that a
measurable function $f$ belongs to
$W(L^p,L^q)$ if the following norm
\begin{equation}\label{wienerdis}\|f\|_{W(L^p,L^q)}= \left(\sum_{n\in\zd}\left(\intrd
|f(x)T_n\chi_\mathcal{Q}(x)|^p\right)^{\frac
qp}\right)^{\frac 1q},
\end{equation}
where  $\mathcal{Q}=[0,1)^d$
(with the usual adjustments
if $p=\infty$ or $q=\infty$)
is finite (see \cite{Heil03}
and Section \ref{section2}
below). In particular,
$W(L^p,L^p)=L^p$. For
heuristic purposes, functions
in $W(L^p,L^q)$ may be
regarded as functions which
are locally in $L^p$ and
decay at infinity like a
function in $L^q$. In this
connection, our results read
as follows (see Theorem
\ref{29-0} and Proposition
\ref{t41bis-1}).\par\medskip
 {\it
Let $1\leq p,q,r,s\leq\infty$ such that
\begin{equation}\label{31-1}
\frac{1}{p}\geq
\left|\frac{1}{r}-\frac{1}{2}\right|+\frac{1}{q'},\qquad
q\leq\min\{r,r',s,s'\}.
\end{equation}
Then every symbol $a\in M^{p,q}$ gives
rise to a bounded operator $a(x,D)$ on
$W(L^r,L^s)$.\\
Viceversa, if this conclusion holds
true, then the constraints in
\eqref{31-1} must be satisfied.}\par
\medskip
Finally, we investigate the
boundedness of $a(x,D)$ on
modulation spaces. We wonder
whether there are results
other than those which follow
by interpolation from the
known ones. It turns out that
this is not the case, as
shown by the following result
(see Theorem \ref{1a} and
Proposition
\ref{1-3}).\par\medskip {\it
Let $1\leq p,q,r,s\leq\infty$
such that
\begin{equation}\label{31-2}
 p\leq q',\qquad
q\leq\min\{r,r',s,s'\}.
\end{equation}
Then every symbol $a\in M^{p,q}$ gives
rise to a
bounded operator $a(x,D)$ on $M^{r,s}$.\\
Viceversa, if this conclusion
holds true, then the
constraints in \eqref{31-2}
must be
satisfied.}\par\medskip This
last result generalizes
\cite{GH04}, where the above
necessary conditions were
proved in the case $r=s=2$
(i.e. for
$L^2$-boundedness).\par So
far we considered
pseudodifferential operators
in the form \eqref{KN}, which
is usually referred to as the
Kohn-Nirenberg
correspondence. However, as
shown in Section
\ref{section4}, all the above
results concerning symbols in
modulation spaces apply to
the Weyl quantization as
well, defined in terms of the
the cross-\Wd\ $W(f,g)$
in~\eqref{eq3232} by $\la
L_\sigma f,g\ra=\la
\sigma,W(g,f)\ra,\
f,g\in\sch(\Ren)$, or
directly as
\begin{equation}\label{equiv1}
L_\sigma f(x)=\int e^{2\pi
i(x-y)\omega}\sigma\left(\frac{x+y}{2},\omega\right)
f(y)\,dy\,d\omega.
\end{equation}
Instead, we shall prove in
Section \ref{section6} that,
contrary to what happens for
modulation spaces, Wiener
amalgam spaces $W(\Fur
L^p,L^q)$, for $p\not=q$ are
not invariant under the
action of the mateplectic
operator which switches the
Kohn-Nirenberg and Weyl
symbol of a
pseudodifferential operator,
so that the above mentioned
conunterexamples for symbols
in $W(\Fur L^1,L^\infty)$
will be provided for both
Kohn-Nirenberg and Weyl
operators.
%Every continuous operator from $\cS (\rd )$ to $\cS ' (\rd )$ can be representated as a Weyl transform.

\par\medskip\noindent
The paper is organized as
follows. Section
\ref{section2} is devoted to
preliminary definitions and
properties of the involved
function spaces. In Sections
\ref{section3} and
\ref{section4} we study
sufficient and necessary
conditions, respectively, for
the boundedness on Wiener
amalgam spaces (results for
the Lebesgue spaces are
attained there as a
particular case). Section
\ref{section5} provides
necessary and sufficient
conditions for the
boundedness on modulation
spaces.
 Finally Section
 \ref{section6} is devoted to
 some result for operators with
 symbols in $W(\Fur
 L^p,L^q)$.

\vskip0.5truecm \textbf{Notation.} To
be definite, let us fix some notation
we shall use later on (and have already
used in this Introduction). We define
$xy=x\cdot y$, the scalar product on
$\Ren$. We define by
$\cC^\infty_0(\rd)$ the space of smooth
functions on $\rd$ with compact
support. The Schwartz class is denoted
by  $\sch(\Ren)$, the space of tempered
distributions by  $\sch'(\Ren)$.   We
use the brackets $\la f,g\ra$ to denote
the extension to $\sch (\Ren)\times\sch
'(\Ren)$ of the inner product $\la
f,g\ra=\int f(t){\overline {g(t)}}dt$
on $L^2(\Ren)$.  The Fourier transform
is normalized to be ${\hat
  {f}}(\o)=\Fur f(\o)=\int
f(t)e^{-2\pi i t\o}dt$.
Moreover we set
$f^\ast(x)=f(-x)$. Throughout
the paper, we shall use the
notation $A\lesssim B$,
$A\gtrsim B$ to indicate
$A\leq c B$, $A\geq c B$
respectively, for a suitable
constant $c>0$, whereas $A
\asymp B$ if $A\leq c B$ and
$B\leq k A$, for suitable
$c,k>0$.

\section{Preliminary
results}\label{section2}
\subsection{Function Spaces}\label{FS}
For $1\leq p\leq \infty$, recall the $\cF L^p$ spaces, defined by
$$\cF L^p(\rd)=\{f\in\cS'(\rd)\,:\, \exists \,h\in L^p(\rd),\,\hat h=f\};
$$
they are Banach spaces equipped with the norm
\begin{equation}\label{flp}
\| f\|_{\cF L^p}=\|h\|_{L^p},\quad\mbox{with} \,\hat h=f.
\end{equation}
\par

The   mixed-norm space $L^{p,q}(\rdd)$, $1\leq p,q\leq\infty$, consists of all measurable functions on $\rdd$ such that the norm
\begin{equation}\label{normamista}\|F\|_{L^{p,q}}=\left(\intrd\left(\intrd |F\phas|^p dx\right)^{\frac qp} d\o\right)^{\frac1q}
\end{equation}
(with obvious modifications when $p=\infty$ or $q=\infty$)
is finite.
\par

The function spaces  $L^pL^q(\rdd)$, $1\leq p,q\leq\infty$, consists of all measurable functions on $\rdd$ such that the norm
\begin{equation}\label{normainversa}\|F\|_{L^pL^q}=\left(\intrd\left(\intrd |F\phas|^q d\o\right)^{\frac pq} dx\right)^{\frac1q}
\end{equation}
(with obvious modifications when $p=\infty$ or $q=\infty$)
is finite. Notice that, for $p=q$, we have  $L^pL^p(\rdd)=L^{p,p}(\rdd)=L^p(\rdd)$.

\textbf{Wiener amalgam spaces}. We
briefly recall the definition and the
main properties of Wiener amalgam
spaces. We refer to
\cite{feichtinger80,feichtinger83,fournier-stewart85,Heil03}
for details.\par
 Let $g \in
\cC_0^\infty$ be a test function that
satisfies $\|g\|_{L^2}=1$. We will
refer to $g$ as a window function. Let
$B$ one of the following Banach spaces:
$L^p, \cF L^p$, $L^{p,q}$, $L^pL^q$,
$1\leq p,q\leq \infty$. Let $C$ be one
of the following Banach spaces: $L^p$,
$L^{p,q}$, $L^pL^q$, $1\leq
p,q\leq\infty$. For any given temperate
distribution $f$ which is locally in
$B$ (i.e. $g f\in B$, $\forall
g\in\cC_0^\infty$), we set $f_B(x)=\|
fT_x g\|_B$.

The {\it Wiener amalgam space} $W(B,C)$
with local component $B$ and global
component  $C$ is defined as the space
of all temperate distributions $f$
locally in $B$ such that $f_B\in C$.
Endowed with the norm
$\|f\|_{W(B,C)}=\|f_B\|_C$, $W(B,C)$ is
a Banach space. Moreover, different
choices of $g\in \cC_0^\infty$ generate
the same space and yield equivalent
norms.

If  $B=\Fur L^1$ (the Fourier
algebra),  the space of
admissible windows for the
Wiener amalgam spaces $W(\Fur
L^1,C)$ can be enlarged to
the so-called Feichtinger
algebra $W(\Fur L^1,L^1)$.
Recall  that the Schwartz
class $\sch$
  is dense in $W(\Fur L^1,L^1)$.\par
   The following properties of
Wiener amalgam spaces  will
be frequently used in the
sequel.
\begin{lemma}\label{WA}
  Let $B_i$, $C_i$, $i=1,2,3$, be Banach spaces  such that $W(B_i,C_i)$ are well
  defined. Then,
  \begin{itemize}
  \item[(i)] \emph{Convolution.}
  If $B_1\ast B_2\hookrightarrow B_3$ and $C_1\ast
  C_2\hookrightarrow C_3$, we have
  \begin{equation}\label{conv0}
  W(B_1,C_1)\ast W(B_2,C_2)\hookrightarrow W(B_3,C_3).
  \end{equation}
  \item[(ii)]\emph{Inclusions.} If $B_1 \hookrightarrow B_2$ and $C_1 \hookrightarrow C_2$,
   \begin{equation*}
   W(B_1,C_1)\hookrightarrow W(B_2,C_2).
  \end{equation*}
  \noindent Moreover, the inclusion of $B_1$ into $B_2$ need only hold ``locally'' and the inclusion of $C_1 $ into $C_2$  ``globally''.
   In particular, for $1\leq p_i,q_i\leq\infty$, $i=1,2$, we have
  \begin{equation}\label{lp}
  p_1\geq p_2\,\mbox{and}\,\, q_1\leq q_2\,\Longrightarrow W(L^{p_1},L^{q_1})\hookrightarrow
  W(L^{p_2},L^{q_2}).
  \end{equation}
  \item[(iii)]\emph{Complex interpolation.} For $0<\theta<1$, we
  have
\[
  [W(B_1,C_1),W(B_2,C_2)]_{[\theta]}=W\left([B_1,B_2]_{[\theta]},[C_1,C_2]_{[\theta]}\right),
  \]
if $C_1$ or $C_2$ has absolutely
continuous norm. The same holds if
every Wiener amalgam space is replaced
by the closure of the Schwartz space
into itself.\par\noindent
    \item[(iv)] \emph{Duality.}
    If $B',C'$ are the topological dual spaces of the Banach spaces $B,C$ respectively, and
    the space of test functions $\cC_0^\infty$ is dense in both $B$ and $C$, then
\begin{equation}\label{duality}
W(B,C)'=W(B',C').
\end{equation}
 \item[(v)] \emph{Pointwise products.}
  If $B_1\cdot B_2\hookrightarrow B_3$ and $C_1\cdot
  C_2\hookrightarrow C_3$, we have
  \begin{equation}\label{point0}
  W(B_1,C_1)\cdot W(B_2,C_2)\hookrightarrow W(B_3,C_3).
  \end{equation}
  \end{itemize}
  \end{lemma}
Finally, recall the
following result, proved in
\cite[Proposition
2.7]{CNloc09}.
\begin{lemma}\label{new}
Let $1\leq q\leq
p\leq\infty$. For every
$R>0$, there exists a
constant $C_R>0$ such that,
for every
$f\in\mathcal{S}'(\R^d)$
whose Fourier transform is
supported in any ball of
radius $R$, it turns out
\[
\|f\|_{W(L^p,L^q)}\leq
C_R\|f\|_{q}.
\]
\end{lemma}
\subsection{ Short-Time Fourier Transform (STFT) and Wigner distribution}
The \tfr s needed for our results are the \emph{\stft } and the \emph{\Wd }.

The \stft\   (STFT)  of a distribution $f\in\sch'(\Ren)$ with
respect to a non-zero window $g\in\sch(\Ren)$ is
\begin{equation}\label{eqi2} V_gf(x,\o)=\la f,M_\o T_x g\ra =\int_{\Ren}
 f(t)\, {\overline {g(t-x)}} \, e^{-2\pi i\o t}\,dt\, ,
 \end{equation}
whereas the  {\it cross-Wigner distribution} $W(f,g)$ of $f,g\in
L^2(\Ren)$ is defined to be
\begin{equation}
  \label{eq3232}
  W(f,g)(x,\o)=\int f(x+\frac{t}2)\overline{g(x-\frac{t}2)} e^{-2\pi
    i\o t}\,dt.
\end{equation}
The quadratic expression $Wf = W(f,f)$ is usually called the Wigner
distribution of $f$.

Both the STFT  $\vgf $ and the \Wd\ $W(f,g)$ are defined on many pairs
of Banach spaces.  For instance, they both map $L^2(\rd ) \times
L^2(\rd )$ into $\lrdd $ and  $\sch(\Ren)\times\sch(\Ren)$ into
$\sch(\Renn)$. Furthermore, they  can be extended  to a map from
$\sch'(\Ren)\times\sch(\Ren)$ into $\sch'(\Renn)$.
We first list some crucial    properties of the STFT  (for proofs, see, e.g.,
\cite[Ch.~3]{book}.

\begin{lemma}\label{propstft}
Let $f,g\in L^2(\Ren)$, then we have
%\item{({\it i})} (Switching f and g),
%$\, (V_fg)(x,\xi)=e^{-2\pi i\xi x}\overline{(V_gf)(-x,-\xi)}. $
%\item{({\it ii})} (Inversion formula),
%$\,\int\int_{\Ren\times\Ren} (V_gf)(x,\xi)M_\xi T_xh\,dxd\xi=\la h,g\ra f. $
%\item{({\it iii})} (Orthogonality relations),
%$$ \la V_{g_1}f_1,V_{g_2}f_2\ra_{L^2(\Ren\times\Ren)}=\la f_1,f_2 \ra_{L^2(\Ren)}{\overline{\la %g_1,g_2 \ra}}_{L^2(\Ren)}.
%$$

(i)\label{STFTform} $\quad V_gf(x,\o)=(f\cdot T_x {\bar
    g})\,\widehat{}\,(\o) $.
%    \begin{equation}
%      \label{eql1}
%V_gf(x,\o)=(f\cdot T_x {\bar g})\,\widehat{}\,(\o)=e^{-2\pi ix\o
%  }(f\ast(M_\o g)^*)(x)
%    \end{equation}

(ii)   (STFT of time-frequency shifts) For $ y,\xi \in \Ren
  $,  we have
  \begin{equation}
    \label{eql2}
   V_{ {g}}(M_\xi T_y{f})(x,\o)  = e^{-2\pi
     i(\o-\xi)y}(V_gf)(x-y,\o-\xi),
  \end{equation}
    %\item{({\it v})} (STFT of the Fourier transforms),
%$\, (V_{\hat {g}}{\hat{f}})(\eta,y)=e^{-2\pi i\eta y}(V_gf)(-y,\eta). $
%\item{({\it vi})} (Fourier transform of the STFT),
%$\, {\widehat{(V_gf)}}(\eta,y)=e^{2\pi i\eta y}f(-y){\overline {{\hat{g}}(\eta)}}. $

%(iii)  (Fourier transform of a product of STFTs),
%$$ (V_{g_1}f_1{\overline{V_{g_2}f_2}}) \, \widehat{} \,
%(x,\o)=(V_{f_2}f_1{\overline {V_{g_2}g_1}})(-\o,x). $$
\end{lemma}
The following result was proved in~\cite[Lemma~14.5.1]{book}.
\begin{lemma}\label{STFTSTFT}
Let $\Phi=W (\f , \f ) \in\sch(\Renn)$. Then the STFT of $W(g,
f) $ with respect to the window $\Phi $ is given by
\begin{equation}
  \label{eql4}
{ {V}}_\Phi (W(g,f)) (z, \zeta ) =e^{-2\pi i
  z_2\z_2}{\overline{V_{\f
      }f(z_1+\frac{\z_2}2,z_2-\frac{\z_1}2})}V_{\f
  }g(z_1-\frac{\z_2}2,z_2+\frac{\z_1}2)\, ,
\end{equation}
  where $z=(z_1,z_2)\in\rdd$, $\z=(\z_1,\z_2)\in\rdd$.
\end{lemma}

\textbf{Modulation spaces}. For their basic properties we refer to
\cite{F1,book}. \par
Given a non-zero window $g\in\sch(\Ren)$ and $1\leq p,q\leq
\infty$, the {\it
  modulation space} $M^{p,q}(\Ren)$ consists of all tempered
distributions $f\in\sch'(\Ren)$ such that the STFT, defined in \eqref{eqi2}, fulfills $V_gf\in L^{p,q}(\Renn )$. The norm on $M^{p,q}$ is
$$
\|f\|_{M^{p,q}}=\|V_gf\|_{L^{p,q}}=\left(\int_{\Ren}
  \left(\int_{\Ren}|V_gf(x,\o)|^p\,
    dx\right)^{q/p}d\o\right)^{1/p}  \, .
$$
If $p=q$, we write $M^p$ instead of $M^{p,p}$.

$M^{p,q}$ is a Banach space
whose definition is independent of the choice of the window $g$. Moreover,
if $g \in   M^1 \setminus \{0\}$, then
$\|V_gf \|_{L^{p,q}}$ is an
equivalent norm for $M^{p,q}(\Ren)$.

 Among the properties of \modsp s, we record  that
$M^{2,2}=L^2$, and we list the following results.
\begin{lemma}\label{mo} We have\\
 (i) $M^{p_1,q_1}\hookrightarrow M^{p_2,q_2}$, if
$p_1\leq p_2$ and $q_1\leq q_2$. \\
(ii) If $1\leq p,q<\infty$, then $(M^{p,q})'=M^{p',q'}$.\\
(iii) For $1\leq p,q,
p_i,q_i\leq \infty$, $i=1,2$,
with $q_1<\infty$ or
$q_2<\infty$, and
 $$\frac 1{p}=\frac{1-\theta}{p_1}+\frac\theta{p_2},\quad \frac 1{q}=\frac{1-\theta}{q_1}+\frac\theta{q_2}, $$
 we have
 $$
[M^{p_1,q_1}, M^{p_2,q_2}]_{\theta}=
M^{p,q}.$$ The same holds if every
modulation space is replaced by the
closure of the Schwartz space into
itself.\par\noindent (iv) If $1\leq
p_i,q_i\leq \infty$, $i=1,2,3$, with
$$\frac1{p_1}+\frac1{p_2}=1+\frac1{p_3},\quad
\frac1{q_1}+\frac1{q_2}=\frac1{q_3}
$$
then
$$M^{p_1,q_1}\ast M^{p_2,q_2}\hookrightarrow M^{p_3,q_3}.$$
\end{lemma}
Modulation spaces and Wiener amalgam spaces are closely related:
for $p=q$, we have
\begin{equation}\label{Wienermod}\|f\|_{W({\cF
L}^p,L^p)}=\left(\int_{\rd}\int_{\rd}|V_{g}f\phas|^p\, m\phas^p
dx\,d\o\right)^{1/p}\asymp  \|f\|_{M^p}.
\end{equation}
More generally, from a comparing the definitions of $\mpq$ and
$W(\cF L^p, L^q)$, it is obvious that $M^{p,q}=\cF W(\cF L^p,
L^q)$.\par
 To prove the boundedness results for pseudodifferential
operators, we shall write their symbols as superposition of
time-frequency shifts. Namely, we shall use the following STFT
inversion formula (see, e.g., \cite{book,GH99}).
\begin{theorem} If $g\in\cS(\rdd)$ and $\|g\|_2=1$, then
\begin{equation}\label{STFTinv}
a=\int_{\R^{4d}} V_g a(\a,\b)M_\b T_\a g d\a d\b.
\end{equation}
If $a\in M^{p,q}$, with $1\leq p,q<\infty$, then this integral converges in the norm of this space.
If $p=\infty$ or $q=\infty$ then this integral converges weakly.
\end{theorem}
We also recall the following
well-known result (see, e.g.,
\cite{fe89-1,kasso07}).
\begin{lemma}\label{lloc} Let $1\leq p,q\leq
\infty.$\\
(i) For every $u\in
\mathcal{S}'(\rd)$, supported
in a compact set $K\subset
\rd$, we have $u\in
M^{p,q}\Leftrightarrow u\in
\Fur L^q$, and
\begin{equation}\label{loc}
C_K^{-1} \|u\|_{M^{p,q}}\leq
\|u\|_{\cF L^q}\leq C_K
\|u\|_{M^{p,q}},
\end{equation}
where $C_K>0$ depends only on
 $K$.\\
 (ii) For every $u\in \mathcal{S}'(\rd)$,
whose Fourier transform is
supported in a compact set
$K\subset \rd$, we have $u\in
M^{p,q}\Leftrightarrow u\in
L^p$, and
\begin{equation}\label{loc2}
C_K^{-1} \|u\|_{M^{p,q}}\leq\|u\|_{ L^p}\leq C_K\|u\|_{M^{p,q}},
\end{equation}
where $C_K>0$ depends only on
 $K$.\\
\end{lemma}

\section{Boundedness on Wiener amalgam spaces: sufficient
conditions}\label{section3}
To avoid the fact that
$\cS(\R^d)$ is not dense in
$W(L^r,L^s)$, if $r=\infty$
or $s=\infty$, we use the
following definition of the
boundedness of
pseudodifferential operators
$a(x,D)$ on Wiener amalgam
spaces: we say that $a(x,D)$
is bounded from
$W(L^r,L^s)(\R^d)$ to
$W(L^r,L^s)(\R^d)$ if there
exists a constant $C>0$ such
that $\|a(x,D)
f\|_{W(L^r,L^s)} \le
C\|f\|_{W(L^r,L^s)}$, for all
$f \in \cS(\R^d)$.\par Let us
recall the following result
(see, e.g., \cite[Theorem
14.3.5]{book}).
\begin{theorem}\label{teo-scambio} Let T be a continuous linear operator mapping $\cS(\rd)$ into $\cS'(\rd)$. Then there exist tempered
distributions $K,\sigma,a\in\cS'(\rdd)$, such that T has the following representations: \\
(i) as an integral operator $\la T f,g\ra=\la K,g\otimes \bar{f}\ra$, for $f,g\in\cS(\rd)$;\\
(ii) as a pseudodifferential operator $T=L_{\sigma}$, with Weyl symbol $\sigma$ and $T=a(x,D)$ with Kohn-Nirenberg symbol $a$.\\
The relations between $K,\sigma,a$ are given by
\begin{equation}\label{passaggio}
\sigma=\cF_2\tau_s K,\quad\quad a=\cU\sigma,\quad\quad \sigma=\cU^{-1}a
\end{equation}
where $\cF_2$ is the partial Fourier transform in the second variable,  $\tau_s$ is the symmetric coordiante transformation $\tau_s K(x,y)=K(x+\frac y2,x-\frac y 2)$ and the operator $\cU$ is defined by $\widehat{(\cU\sigma)} (\o_1,\o_2)=e^{\pi i \o_1\o_2}\hat{\sigma}(\o_1,\o_2).$
\end{theorem}
\begin{remark}\rm\label{oss-scambio}
Since $a(x,D)=L_{\cU^{-1}a}$,
a straightforward
modification of
\cite[Corollary 14.5.5]{book}
shows that the modulation
spaces $M^{p,q}$, $1\leq
p,q\leq\infty$, are invariant
under $\cU^{-1}$, so that
boundedness results for
pseudodifferential operators
with symbols in modulation
spaces can be obtained using
either the Weyl or the
Kohn-Nirenberg form. In the
sequel, we shall adopt the
operator form which is more
convenient.
\end{remark}
We also need the following
useful remark.

\begin{remark}\label{densglobal}\rm
Observe that, if $1\leq r<\infty$, then
$$\left(\overline{\cS}^{W(L^r,L^\infty)}\right)'=W(L^{r'},L^1),
$$
see, e.g., \cite[Theorem 2.8]{hans85}.
\end{remark}
\begin{theorem}\label{bound1}
If $\sigma\in M^{\infty,1}(\rdd)$, then the Weyl operator $L_\sigma$  is bounded on \\ $W(L^2,L^s)(\rd)$, for every $1\leq s\leq\infty$, with the uniform estimate
$$\|L_\sigma f\|_{W(L^2,L^s)}\lesssim \|\sigma\|_{ M^{\infty,1}}\|f\|_{W(L^2,L^s)}.$$
\end{theorem}
\begin{proof}
Let us show the estimate
$$|\la L_\sigma f,g\ra |\lesssim \|\sigma\|_{M^{\infty,1}}\|f\|_{W(L^2,L^s)}\|g\|_{W(L^2,L^{s'})},\quad \forall f,g\in\cS(\R^d),
$$
where $1/s+1/{s'}=1$. This will give at
once the desired result if $s>1$,
whereas the case $s=1$ follows by Remark \ref{densglobal}.\par Let $\f\in
\mathcal{C}^\infty_0 (\rd )$ and set
$\Phi=W(\f,\f)\in\sch(\Renn)$. By the
definition of the Weyl operator via
cross-Wigner distribution and
H\"older's inequality,
\begin{align*}|\la L_\sigma f,g\ra |&=|\la \sigma, W(g,f)\ra|=|\la V_{\Phi}\sigma, V_{\Phi}W(g,f)\ra|\leq \|V_{\Phi}\sigma\|_{L^{\infty,1}}\|V_{\Phi}W(g,f)\|_{L^{1,\infty}}\\
&\asymp \|\sigma\|_{M^{\infty,1}}\|W(g,f)\|_{M^{1,\infty}}.
\end{align*}
Then, the result is proved if we show that $\|W(g,f)\|_{M^{1,\infty}}\lesssim \|f\|_{W(L^2,L^s)}\|g\|_{W(L^2,L^{s'})}$.
  If $\zeta = (\z_1,\z_2)\in \Renn$, we write $\tilde{\zeta } = (\zeta _2,-\zeta _1)$. Then
Lemma~\ref{STFTSTFT} says that
$$|{ {V}}_\Phi (W(g,f))(z,\zeta)| =| V_\f f(z
+\tfrac{\tilde{\z }}{2})| \,  |V_\f g(z - \tfrac{\tilde{\z }}{2})| \,.$$
Consequently
\begin{equation*}
  \|W(g,f)\|_{M^{1,\infty}} \asymp  \sup_{\z\in\rdd} \intrdd | V_\f f(z +\tfrac{\tilde{\z }}{2})| \,  |V_\f g(z -\tfrac{\tilde{\z }}{2})|   \, dz.
\end{equation*}
We set $\pi(\tilde{\z })=M_{\tilde{\z_2 }}T_{\tilde{\z_1 }}$. In what follows, we make  the change of variables $z \mapsto z-\tilde{\zeta } /2$, and use Lemma \ref{propstft}, (i) and (ii), the Cauchy-Schwarz's and Parseval's inequalities with respect to the $z_2$ variable, so that
\begin{align*}
\|W(g,f)\|_{M^{1,\infty}}
 &\asymp \sup_{\z\in\rdd} \intrdd | V_\f f(z)| \,  |V_\f g(z -\tilde{\z })|   \, dz\\
&=\sup_{\z\in\rdd} \intrdd | V_\f f(z)| \,  |V_\f (\pi(\tilde{\z })g)(z )|   \, dz \\
&=\sup_{\z\in\rdd} \intrd\intrd | \widehat{fT_{z_1}\f}(z_2)|\,|\widehat{\pi(\tilde{\z })gT_{z_1}\f}(z_2)| \,  dz_1 dz_2 \\
&\lesssim \sup_{\z\in\rdd} \intrd  \|fT_{z_1}\f\|_2\,\|\pi(\tilde{\z })gT_{z_1}\f\|_2 \,  dz_1 \\
&\lesssim \sup_{\z\in\rdd} \|f\|_{W(L^2,L^s)}\| \pi(\tilde{\z })g\|_{W(L^2,L^{s'})}\\
&=\|f\|_{W(L^2,L^s)}\|g\|_{W(L^2,L^{s'})},
\end{align*}
where we have used H\"older's inequality  in the last-but-one  step and the invariance of the $W(L^2,L^s)$ spaces under time-frequency shifts $\pi(\tilde{\z })$ in the last one.
\end{proof}

If we choose symbols with a stronger
decay, namely symbols $a\in
M^{p,1}\subset M^{\infty,1}$, $1\leq
p\leq 2$, then the corresponding
pseudodifferential operators $a(x,D)$
are bounded on every Wiener amalgam
spaces $W(L^r,L^s)$, as shown in the
following result.

\begin{theorem}\label{bound2}
If $a\in M^{p,1}(\rdd)$, $1\leq p\leq
2$, then the operator $a(x,D)$  is
bounded on $W(L^r,L^s)(\rd)$, for every
$1\leq r, s\leq\infty$, with the
uniform estimate
$$\| a(x,D)f\|_{W(L^r,L^s)}\lesssim \|a\|_{ M^{p,1}}\|f\|_{W(L^r,L^s)}.$$
\end{theorem}
\begin{proof}
By the inclusion relations for modulation spaces we can just
consider the case  $a\in M^{2,1}$. We shall show that the integral
kernel
\[
K(x,y)=(\Fur_2 a)(x,y-x)
\]
($\Fur_2$  stands for the partial Fourier transform with respect to the
second variable) of $a(x,D)$ can be controlled from above by
$|K(x,y)|\leq F(x-y)$, where  $F$ is a positive function in
$L^1(\rd)$. If it is so, then $|a(x,D)f(x)|\leq (F\ast |f|)(x)$
 and the convolution relations for Wiener amalgam spaces in Lemma
 \ref{WA} (i) give the desired result.\par
 We use the inversion formula  \eqref{STFTinv} for the
symbol $a$. Namely, for any window $g\in\cS(\R^{2d})$, with
$\|g\|_2=1$, we have
\begin{equation}\label{istft}
a(x,\o)=\int_{\R^{4d}} (V_g
a) (\alpha,\beta) (M_\beta
T_\alpha
g)(x,\o)\,d\alpha\,d\beta.
\end{equation}
Hence, if
$\alpha=(\alpha_1,\alpha_2)$,
$\beta=(\beta_1,\beta_2)$, it
turns out
\[
K(x,y)=\int_{\R^{4d}} e^{2\pi i \a_2\b_2}(V_g a)
(\alpha_1,\a_2,\beta_1,\b_2)
M_{(\b_1,\a_2)}T_{(\a_1,-\b_2)}(\Fur_2^{-1} g)(x,x-y)d\alpha_1
d\a_2d\beta_1 d\b_2.
\]
Setting
\[
H(\alpha_1,t;\beta_1,\b_2)=\int_{\R^d}(V_g a) (\alpha,\beta)
e^{2\pi i(t-\b_2)\alpha_2}\,d\alpha_2,
\]
and using the Cauchy-Schwarz's inequality with respect to the $\a_2$ variable, we  obtain
\begin{align*}
|K(x,y)|&\leq \int_{\R^{3d}}
|H(\alpha_1,x-y+\b_2;\beta_1,\b_2)\Fur_2^{-1}g(x-\alpha_1,x-y+\beta_2)|d\alpha_1\,d\beta_1 d\b_2\\
&\leq \intrdd
\|H(\cdot,x-y+\b_2;\b_1,\b_2)\|_2\|T_{(0,-\b_2)}\Fur_2^{-1}
g(\cdot,x-y)\|_2 d\b_1 d\b_2.
\end{align*}
For simplicity, let us set $$F(t):=\intrdd
\|H(\cdot,t+\b_2;\b_1,\b_2)\|_2\|T_{(0,-\b_2)}\Fur_2^{-1}
g(\cdot,t)\|_2 d\b_1 d\b_2,$$ so that $|K(x,y)|\leq F(x-y)$.\\
 We are left to estimate $\| F\|_{1}$. Using the
Cauchy-Schwarz's inequality with respect to the $t$ variable,
\begin{align*}
\|F\|_1 &=\intrd \intrdd
\|H(\cdot,t+\b_2;\b_1,\b_2)\|_2\|T_{(0,-\b_2)}\Fur_2^{-1}
g(\cdot,t)\|_2 d\b_1 d\b_2 dt\\
&\leq\intrdd \|H(\cdot,\cdot;\b_1,\b_2)\|_2\|T_{(0,-\b_2)}\Fur_2^{-1}g\|_2 d\b_1 d\b_2\\
&=\|g\|_2\intrdd \|H(\cdot,\cdot;\b_1,\b_2)\|_2 d\b_1 d\b_2\\
&=\|H\|_{L^{2,1}}=\|V_g a\|_{L^{2,1}}\asymp\|a\|_{M^{2,1}},
\end{align*}
where in the last row we used Parseval's formula and the assumption $\|g\|_2=1$. This concludes the proof.
\end{proof}
\begin{theorem}\label{29-0}
Let $1\leq p,q,r,s\leq\infty$
such that
\begin{equation}\label{limitazionilp}
\frac{1}{p}\geq
\left|\frac{1}{r}-\frac{1}{2}\right|+\frac{1}{q'},\qquad
q\leq\min\{r,r',s,s'\}.
\end{equation}
Then every symbol $a\in M^{p,q}$ gives
rise to a bounded operator $a(x,D)$ on
$W(L^r,L^s)$ with the uniform estimate
\[
\|a(x,D)f\|_{W(L^r,L^s)}\lesssim
\|a\|_{M^{p,q}}\|f\|_{W(L^r,L^s)}.
\]
\end{theorem}
\begin{proof}
We first make the complex interpolation
between the estimates of Theorem
\ref{bound1} and Theorem \ref{bound2}
(which deal with the cases in which
$q=1$). Using the interpolation
relations for Wiener amalgam and
modulation spaces of Lemma \ref{WA}
(iii) and Lemma \ref{mo} (iii), we
obtain, for
 every $1\leq s<\infty$,
$$\|a(x,D) f\|_{W(L^r,L^s)}\lesssim \|a\|_{M^{p,1}}\|f\|_{W(L^r,L^s)},$$
where
$$
\frac{1}{p}\geq
\left|\frac{1}{r}-\frac{1}{2}\right|.
$$
The remaining cases, when  $s=\infty$, $p>2$ (and therefore
$r>1$), follow by duality, for $W(L^r,L^\infty)=W(L^{r'},L^1)'$
(Lemma \ref{WA} (iv)) and, considering the Weyl form $L_\sigma$ of
$a(x,D)$, we have $(L_\sigma)^*=L_{\bar{\sigma}}$.
\par Finally,
by interpolation between what we just
proved and the well-known case
$p=q=r=s=2$ (pseudodifferential
operators with symbols in
$M^2(\rdd)=L^2(\rdd)$ are bounded on
$W(L^2,L^2)(\rd)=L^2(\rd)$; see
\cite[Theorem 14.6.1]{book}), we obtain
the claim.
\end{proof}

Recalling that, for $s=r$, we have
$W(L^r,L^r)=L^r$, the above boundedness
result can be rephrased for
pseudodifferential operators acting on
$L^p$ spaces as follows (see Figure 1
in Introduction).
\begin{corollary}\label{29-1}
Let $1\leq p,q,r\leq\infty$
such that
\begin{equation}\label{limitazionilpr}
\frac{1}{p}\geq
\left|\frac{1}{r}-\frac{1}{2}\right|+\frac{1}{q'},\qquad
q\leq\min\{r,r'\}.
\end{equation}
Then every symbol in $a\in
M^{p,q}$ gives rise to a
bounded operator $a(x,D)$ on $L^r$
with the uniform estimate
\[
\|a(x,D)f\|_r\lesssim
\|a\|_{M^{p,q}}\|f\|_{r}.
\]
\end{corollary}
\section{Boundedness on Wiener amalgam spaces: necessary
conditions}\label{section4} In this
section we show the optimality of
Theorem \ref{29-0} (and Corollary
\ref{29-1}). We need the following
auxiliary results.
\begin{lemma}\label{dualita1}
Suppose that for some $1\leq
p,q,r,s\leq\infty$ the
following estimate holds:
\[
\|L_\sigma
f\|_{W(L^r,L^s)}\leq
C\|\sigma\|_{M^{p,q}}\|f\|_{W(L^r,L^s)},\
\  \forall \sigma\in
\cS(\R^{2d}),\ \forall f\in
\cS(\R^d).\] Then the same
estimate is satisfied with
$r,s$ replaced by $r',s'$
(even if $r=\infty$ or
$s=\infty$).
\end{lemma}
\begin{proof}
Indeed, observe that $\langle
L_\sigma f,g\rangle=\langle
f,L_{\overline{\sigma}}g\rangle$,
$\forall f,g\in\cS(\R^d)$.
Hence, by Lemma \ref{WA} (iv)
and  the assumptions written
for $L_{\overline{\sigma}}$
(observe that
$\|\sigma\|_{M^{p,q}}=\|\bar{\sigma}\|_{M^{p,q}}$),
we have
\begin{equation}\label{dua1}
|\langle L_\sigma
f,g\rangle|\leq
C\|\sigma\|_{M^{p,q}}\|f\|_{W(L^{r'},L^{s'})}\|g\|_{W(L^{r},L^{s})},\quad\forall
f\in \cS(\R^d),\ \forall g\in
\cS(\R^d).
\end{equation}
Since $L_\sigma f$ is a
Schwartz function, it belongs
to $W(L^{r'},L^{s'})\subset
W(L^{r},L^{s})'$, and
$\|L_\sigma
f\|_{W(L^{r},L^{s})'}=\|L_\sigma
f\|_{W(L^{r'},L^{s'})}$,
because $W(L^{r'},L^{s'})$ is
isometrically embedded in
$W(L^r,L^s)'$. Hence it
suffices to prove that the
estimate in \eqref{dua1}
holds for every $g\in
W(L^r,L^s)$. This follows by
a density argument. Namely,
consider, for a given $g\in
W(L^r,L^s)$, a sequence $g_n$
of Schwartz functions, with
$g_n\to g$ in $\cS'(\R^d)$
and $\|g_n\|_{W(L^r,L^s)}\leq
\|g\|_{W(L^r,L^s)}$
(\footnote{For example, take
$g_n(x)=n^d\f_1(x/n)\left(
g\ast\f_2(n\,\cdot)\right)(x)$,
with
$\f_1,\f_2\in\cC^\infty_0(\R^d)$,
$\f_1(0)=1$,
$\|\f_2\|_1=1$.}). Letting
$n\to\infty$ in the above
estimate (written with $g_n$
in place of $g$) gives the
desired conclusion.
\end{proof}
\begin{lemma}\label{l1} Let
$h\in\cC^\infty_0(\R^d)$, and
consider the family of
functions
\[
h_\lambda(x)=h(x) e^{-\pi
i\lambda|x|^2}, \qquad
\lambda\geq1.
\]
Then, for $1\leq
q\leq\infty$,
\begin{equation}\label{ub}
\|\widehat{h_\lambda}\|_{q}\asymp
\lambda^{\frac{d}{q}-\frac{d}{2}}.
\end{equation}
\end{lemma}
\begin{proof} The result is
known and outlined, e.g., in
\cite[Exercise 2.34]{tao}. We report on
a sketch of the proof for the sake of
completeness.\par Let $c>0$ be such
that $h(x)$ vanishes for $|x|>c$. First
one shows the estimate
$|\widehat{h_\lambda}(\o)|\leq
C_N\langle \o\rangle^{-N}\lambda^{-N}$,
for every $N>0$ and $\o\in\rd$ such
that $|\o|\geq 2c\lambda$. To this end,
we observe that by rotational symmetry
we can assume $\o=(\o_1,0,...,0)$. The
claim then follows by applying the
Non-stationary Phase Theorem
\cite[Proposition 1, page 331]{stein93}
with  the asymptotic parameter $\o_1$
and the phase $\phi(x_1):=-2\pi
x_1-\pi\frac{\lambda}{\o_1}x_1^2$ (the
assumptions being satisfied for
$|x_1|\leq c$, uniformly with respect
to the parameter $\lambda/\o_1$). \par
In the region $|\o|<2c\lambda$ we have
the estimate
$|\widehat{h_\lambda}(\o)|\leq
C\lambda^{-d/2}$, as a consequence of
the Stationary Phase Theorem (see
\cite[5.13 (a), page 363]{stein93})
with the phase given by the quadratic
polynomial $\phi(x):=-\pi |x|^2$. One
hence obtains the upper bound
$\|\widehat{h_\lambda}\|_q\lesssim
\lambda^{d(1/q-1/2)}$. Since
\[
\|h\|_2^2=\|\widehat{h_\lambda}\|_2^2\leq\|\widehat{h_\lambda}\|_q\|\widehat{h_\lambda}\|_{q'}\lesssim
\lambda^{d(1/q'-1/2)}\|\widehat{h_\lambda}\|_q,
\]
the lower bound follows as
well.
\end{proof}

We now establish a version of
the upper bound in Lemma
\ref{l1}, for Wiener amalgam
spaces.
\begin{lemma}\label{l2} With the notation of
Lemma \ref{l1} we have, for
$1\leq q\leq\infty$,
\[
\|\widehat{h_\lambda}\|_{W(L^p,L^q)}\lesssim
\lambda^{\frac{d}{q}-\frac{d}{2}},\qquad
\lambda\geq1.
\]
\end{lemma}
\begin{proof}
When $p\geq q$ the desired
result follows from Lemmata
\ref{new} and \ref{l1}. When
$p< q$ the result follows
from the inclusion
$L^q\hookrightarrow
W(L^p,L^q)$ and Lemma
\ref{l1}.
\end{proof}
\begin{proposition}\label{t41} Let
$\chi\in\cC^\infty_0(\R^d)$,
$\chi\geq0$, $\chi(0)=1$.
Suppose that, for some $1\leq
p,q,r,r_1,r_2\leq\infty$,
$C>0$, the estimate
\begin{equation}\label{stimacont}
\| \chi a(x,D)f\|_r\leq C
\|a\|_{M^{p,q}}\|f\|_{W(L^{r_1},L^{r_2})},\qquad
\forall a\in\cS(\R^{2d}),\
f\in\cS(\R^d),
\end{equation}
holds. Then $q\leq r'_2$ and
\begin{equation}\label{limitazionilp2}
\frac{1}{p}\geq\frac{1}{2}-\frac{1}{r}+\frac{1}{q'}.
\end{equation}
\end{proposition}
\begin{proof}
First we prove the constraint $q\leq r'_2$. Let $h\in
\cC^\infty_0(\R^d)$, $h\geq0$, $h(0)=1$. Let $h_\lambda$ be as in
Lemma \ref{l1}.
 We test \eqref{stimacont}
on the family of symbols
\[
a_\lambda(x,\o)=h(x)  h_\lambda(\o)=h(x) h(\o) e^{-\pi
i\lambda|\o|^2},
\]
and functions $ f_\lambda=\Fur^{-1}
\left(\overline{h_\lambda}\right)\in \cS(\rd)$. An explicit
computation shows that
\[
\chi(x) a_\lambda(x,D) f_\lambda(x)=\int e^{2\pi ix\o}\chi(x) h(x)
h^2(\o)\,d\o,
\]
which is a non-zero Schwartz function independent of $\lambda$. On
the other hand, by Lemma \ref{l1}, we have
\[
\|a_\lambda\|_{M^{p,q}}\asymp
\|a_\lambda\|_{\Fur
L^q}\lesssim\|h_\lambda\|_{\Fur
L^q}\lesssim
\lambda^{\frac{d}{q}-\frac{d}{2}}.
\]
Similarly, by Lemma \ref{l2},
\[
\|f_\lambda\|_{W(L^{r_1},L^{r_2})}\lesssim
\lambda^{\frac{d}{r_2}-\frac{d}{2}}.
\]
Taking into account these
estimates and letting
$\lambda\to+\infty$,
\eqref{stimacont} then gives
$q\leq r'_2$.\par Let us now
prove \eqref{limitazionilp2}.
Let $h_\lambda$ be as above.
We now test the estimate
\eqref{stimacont} on the
family of symbols
\[
a'_\lambda(x,\o)= e^{-\pi\lambda|x|^2} \widehat{h_\lambda}(\o),
\]
and functions
$f'_\lambda=\overline{h_\lambda}$.
The operator
$a'_\lambda(x,D)$ has
integral kernel
\[
K_\lambda(x,y)=e^{-\pi
\lambda^2|x|^2}
h_\lambda(x-y),
\]
so that
\begin{align*}
\chi(x)|a'_\lambda(x,D)f'_\lambda(x)|&=\left|\int e^{-\pi
\lambda^2|x|^2+2\pi i\lambda xy} h(x-y)\chi(x)
h(y)\,dy\right|\\
&\geq {\rm Re}\int e^{-\pi \lambda^2|x|^2+2\pi i\lambda xy}
h(x-y)\chi(x) h(y)\,dy.
\end{align*}
Now, $h(y)$ has compact
support, say, in the ball
$|y|\leq C$. Moreover, if
$|x|\leq \lambda^{-1}$ for
$\lambda\geq\lambda_0$ large
enough, and $|y|\leq C$ we
have ${\rm Re}\left(e^{-\pi
\lambda^2|x|^2+2\pi i\lambda
xy}\right)\geq\frac{1}{2}$.
Hence we deduce
\[
\chi(x)|a'_\lambda(x,D)f'_\lambda(x)|\gtrsim 1,\quad {\rm for}\
|x|\leq \lambda^{-1},
\]
which implies
\[
\|\chi a'_\lambda(x,D)f'_\lambda\|_r\gtrsim
\lambda^{-\frac{d}{r}}.
\]
On the other hand,
$\|f'_\lambda\|_{W(L^{r_1},L^{r_2})}$
is clearly independent of
$\lambda$. Moreover, by
\cite[Lemma 3.2]{cordero2},
\[
\|e^{-\pi\lambda^2|\cdot|^2}\|_{M^{p,q}}\lesssim
\lambda^{-\frac{d}{q'}} \]
 and by Lemmata \ref{lloc} (ii)
 and
\ref{l1}  we have
\[
\|\widehat{h_\lambda}\|_{M^{p,q}}\lesssim\|\widehat{h_\lambda}\|_p\lesssim
\lambda^{\frac{d}{p}-\frac{d}{2}}.
\]
Hence
\[
\|a'_\lambda\|_{M^{p,q}}=\|e^{-\pi\lambda^2|\cdot|^2}\|_{M^{p,q}}
\|\widehat{h_\lambda}\|_{M^{p,q}}\lesssim\lambda^{\frac{d}{p}-\frac{d}{2}-\frac{d}{q'}}.
\]
\par Putting all together and letting
$\lambda\to\infty$ we obtain
\eqref{limitazionilp2}.
\end{proof}
\begin{proposition}\label{t27}
Suppose that, for some $1\leq
p,q,r,s,r_1,r_2\leq\infty$,
$C>0$, the estimate
\begin{equation}\label{stimacont27}
\|a(x,D)f\|_{W(L^{r_1},L^{r_2})}\leq
C
\|a\|_{M^{p,q}}\|f\|_{W(L^{r},L^{s})},\qquad
\forall a\in\cS(\R^{2d}),\
f\in\cS(\R^d),
\end{equation}
holds. Then $q\leq r$.
\end{proposition}
\begin{proof}
 Let
$h_1,h_2$ be two Schwartz
functions in $\R^d$ such that
$h_1$ and $\widehat{h}_2$ are
real valued, with $h_1(0)=1$,
$\widehat{h_2}(0)=1$, and
satisfying
\begin{equation}\label{h12}
{\rm
supp}\,\widehat{h_1}\subset
B(0,1),\qquad {\rm
supp}\,\widehat{h_2}\subset
B(0,1).
\end{equation}
Consider then, for every
$N\geq1$, the finite lattice
\begin{equation}\label{reticolo0}
\Lambda_N=\{n=(n_1,...n_d)\in4\zd:\
0\leq n_j\leq 4(N-1),\
j=1,...,d\}. \end{equation}
Observe that $\Lambda_N$ has
cardinality $N^d$, and
\begin{equation}\label{reticolo}
|n|\leq d^{1/2}4(N-1),\
\forall n\in\Lambda_N\ \ {\rm
and}\ \ |n-m|\geq4,\ \forall
n,m\in\Lambda_N, n\not=m.
\end{equation}
Moreover, let $h$ be a smooth
real-valued function,
$h\geq0$, $h(0)=1$, supported
in a ball $B(0,\epsilon)$,
for a small $\epsilon$ to be
chosen later. We test the
estimate \eqref{stimacont27}
on the family of functions
$f_N(x)=h(Nx)$ and symbols
\begin{equation}\label{27-0}
a_N(x,\xi)=\sum_{n\in\Lambda_N}b_n(x,\o),\
{\rm where}\
b_n(x,\o)=(M_{-n}h_1)(x)
(T_n h_2)(\o).
\end{equation}
The integral kernel of the
operator $a_N(x,D)$ is given
by
\[
K_N(x,y)=(\Fur^{-1}_2
a_N)(x,x-y)=\sum_{n\in\Lambda_N}\Fur^{-1}_2(b_n)(x,x-y)=\sum_{n\in\Lambda_N}
e^{-2\pi i n y} h_1(x)
\widehat{h_2}(y-x).
\]
We now show that, for a
suitable $\delta>0$,
\begin{equation}\label{27-1}
|a_N(x,D) f_N(x)|\gtrsim 1,\
{\rm for}\ x\in B(0,\delta),
\end{equation}
which implies
\begin{equation}\label{27-2}
\|a_N(x,D)
f_N\|_{W(L^{r_1},L^{r_2})}\gtrsim
1.
\end{equation} In order to
prove \eqref{27-1} observe
that, by the above
computation,
\[
a_N(x,D) f_N(x)=\int_{\rd}
\left(\sum_{n\in\Lambda_N}e^{-2\pi
i n y}\right) h_1(x)
\widehat{h_2}(y-x)h(Ny)\,dy.
\]
Now, as a consequence of the
first condition in
\eqref{reticolo}, we see that
on the support of $h(Ny)$,
hence where $|y|\leq\epsilon
N^{-1}$, we have
\[
{\rm Re}\left(e^{-2\pi i n
y}\right)\geq\frac{1}{2},\quad
\forall n\in\Lambda_N, \] if
$\epsilon\leq d^{-1/2}/24$.
 This implies
that
\[
|a_N(x,D) f_N(x)|\geq
\frac{N^d}{2}\int_{\rd}
h_1(x)
\widehat{h_2}(y-x)h(Ny)\,dy.
\]
Since
$h_1(0)=\widehat{h_2}(0)=1$,
if $\delta$ and $\epsilon$
are small enough, so as
$h_1(x)\geq1/2$ and
$\widehat{h_2}(y-x)\geq{1/2}$
for $|y|\leq\epsilon$,
$|x|\leq\delta$. It turns out
that
\[
 |a_N(x,D)
f_N(x)|\geq
\frac{N^d}{8}\int_{\rd}
h(Ny)\,dy=\frac{\|h\|_1}{8},\qquad{\rm
for}\ |x|\leq\delta,
\]
which implies
\eqref{27-1}.\par We now
prove that
\begin{equation}\label{27-4}
\|a_N\|_{M^{p,q}}\lesssim
N^{d/q}.
\end{equation}
To see this, observe that
\begin{equation}\label{27-4b}
\widehat{b_n}(\zeta_1,\zeta_2)
=(T_{-n}\widehat{h_1})(\zeta_1)(M_{-n}\widehat{h_2})
(\zeta_2).
\end{equation}
We choose a window $\Phi=\f\otimes\f$,
where $\hat{\f}$ is a Schwartz function
supported in $B(0,1)$. It follows from
\eqref{27-4b}, the second condition in
\eqref{reticolo} and \eqref{h12}, that
the functions
$V_{\Phi}{b_n}(z_1,z_2,\omega_1,\omega_2)=
\left(\widehat{b_n}\ast
M_{-z}\widehat{\Phi}^\ast\right)(\omega)$
(with $z=(z_1,z_2), $
$\omega=(\omega_1,\omega_2)$, $\Phi^\ast(z)=\overline{\Phi(-z)}$), vanish
unless $\omega_1\in B(-n,2)$,
$\omega_2\in B(0,2)$. Hence, when $n$
varies in $\Lambda_N$, they have
pairwise disjoint supports, as well as
the functions
$\|V_{\Phi}{b_n}(\cdot,\cdot,\omega_1,\omega_2)\|_{p}$.
We deduce that
\begin{align}
\|{a_N}\|_{M^{p,q}}
&\asymp\left(\int_{\R^{2d}}
\|\sum_{n\in\Lambda_N}V_{\Phi}{b_n}(\cdot,\cdot,\omega_1,\omega_2)\|_{p}^q\,d
\omega_1\,d\omega_2\right)^{1/q}\nonumber\\
&=\left(\int_{\R^{d}}
\left(\sum_{n\in\Lambda_N}\|V_{\Phi}{b_n}(\cdot,\cdot,\omega_1,\omega_2)
\|_{p}^p\right)^{q/p}\,d
\omega_1\,d\omega_2\right)^{1/q}\nonumber\\
&= \left(
\int_{\R^{2d}}\sum_{n\in\Lambda_N}
\|V_{\Phi}{b_n}(\cdot,\cdot,\omega_1,\omega_2)\|_{p}^q\,d
\omega_1\,d\omega_2 \right)^{1/q}\nonumber\\
&= \left(\sum_{n\in\Lambda_N}
\int_{B(-n,2)}\int_{B(0,2)}
\|V_{\Phi}{b_n}(\cdot,\cdot,\omega_1,\omega_2)\|_{p}^q\,d
\omega_1\,d\omega_2
\right)^{1/q}.\label{27-3}
\end{align}
On the other hand, since
$V_{\Phi}{b_n}(z,\omega)=e^{-2\pi
iz\omega} \left({b_n}\ast
M_{\omega}{\Phi}^\ast\right)(z)$,
by Young's inequality, we
have
\[
\|V_{\Phi}{b_n}(\cdot,\cdot,\omega_1,\omega_2)\|_{p}\leq
\|b_n\|_p,
\]
and the expression for $b_n$
in \eqref{27-0} shows that
$\|b_n\|_p$ is in fact
independent of $n$. Hence the
expression in \eqref{27-3} is
\[
\lesssim
\left(\sum_{n\in\Lambda_N}
\int_{B(-n,2)}\int_{B(0,2)}
\,d z_1\,dz_2
\right)^{1/q}=C_d N^{d/q},
\]
which gives \eqref{27-4}.\par
Finally, since the functions
$f_N$ are supported in a
fixed compact subset, we have
\begin{equation}\label{27-5}
\|f_N\|_{W(L^r,L^s)}\asymp
\|f_N\|_r=\|h\|_r N^{-d/r}.
\end{equation}
Combining this estimate with
\eqref{27-2}, \eqref{27-4}
and \eqref{stimacont27}, and
letting $N\to\infty$ yields
the desired constraint $q\leq
r$.
\end{proof}

\begin{theorem}\label{t41bis}
Suppose that, for some $1\leq
p,q,r,s\leq\infty$, the
estimate
\begin{equation}\label{stimacontbis}
\|a(x,D)f\|_{W(L^{r},L^{s})}\leq
C
\|a\|_{M^{p,q}}\|f\|_{W(L^{r},L^{s})},\qquad
\forall a\in\cS(\R^{2d}),\
f\in\cS(\R^d),
\end{equation}
holds. Then $q\leq
\min\{r,r',s,s'\}$ and
\begin{equation}\label{1-1}
\frac{1}{p}\geq\left|\frac{1}{2}-\frac{1}{r}\right|+\frac{1}{q'}.
\end{equation}
\end{theorem}
\begin{proof} We already know
 from Proposition \ref{t27} that
 $q\leq r$. The constraints
 $q\leq r'$ follows by
 duality arguments, namely by Lemma \ref{dualita1}. \par
 Now, we shall prove that
\begin{equation}\label{28-0}
q\leq s',\qquad
\frac{1}{p}\geq
\frac{1}{2}-\frac{1}{r}+\frac{1}{q'}.
\end{equation}
The remaining  constraints will follow by
 duality as above.\par
Let
$\chi\in\cC^\infty_0(\R^d)$,
with $0\leq \chi\leq 1$,
$\chi(0)=1$. Then
\eqref{stimacontbis} implies
\[
\|\chi a(x,D)
f\|_{W(L^r,L^s)}\leq C
\|a\|_{M^{p,q}}\|f\|_{W(L^{r},L^{s})},\,\quad
\forall f\in \cS(\R^d),\
\forall a\in\cS(\R^{2d}).
\]
Since for functions $u$
supported in a fixed compact
subset we have
$\|u\|_{W(L^{r},L^{s})}\asymp
\|u\|_{r}$, we deduce
\[
\|\chi a(x,D) f\|_{r}\leq C
\|a\|_{M^{p,q}}\|f\|_{W(L^{r},L^{s})}\,\quad
\forall f\in \cS(\R^d),\
\forall a\in\cS(\R^{2d}).
\]
Then Proposition \ref{t41} implies
\eqref{28-0}.
\end{proof}
\begin{proposition}\label{t41bis-1}
Suppose that, for some $1\leq
p,q,r,s\leq\infty$, every symbol $a\in
M^{p,q}$ gives rise to an operator
$a(x,D)$ bounded on $W(L^r,L^s)$. Then
the constraints $q\leq
\min\{r,r',s,s'\}$ and \eqref{1-1} must
hold.
\end{proposition}
\begin{proof}
Let $\mathcal{W}(L^r,L^s)$ be the
closure of $\cS(\R^d)$ in
${W}(L^r,L^s)$. By assumption the map
\[
M^{p,q}\ni a\longmapsto a(x,D)\in
B(\mathcal{W}(L^r,L^s),{W}(L^r,L^s))
\]
is well defined. By an application of
the Closed Graph Theorem and Theorem
\ref{t41bis} we see that the desired
conclusion follows if we prove that
this map has closed graph. To this end,
let $a_n\to a$ in $M^{p,q}$, with
$a_n(x,D)\to A$ in
$B(\mathcal{W}(L^r,L^s),{W}(L^r,L^s))$.
We have to prove that $A=a(x,D)$, i.e.
$\langle A
f,g\rangle=\la a(x,D)f,g\rangle$ $\forall
f,g\in\cS(\R^d)$. Now, clearly $
\langle a_n(x,D) f,g\rangle\to \langle
Af,g\rangle$. On the other hand $
\langle a_n(x,D) f,g\rangle=\langle
a_n, G\rangle$, where $G(x,\omega)=
e^{-2\pi i x\omega}
\overline{\hat{f}(\omega)} g(x)$ is a
fixed Schwartz function. Hence $\langle
a_n(x,D) f,g\rangle$ tends to $\langle
a, G\rangle=\langle a(x,D) f,g\rangle$,
which concludes the proof.
\end{proof}
\section{Boundedness on
modulation spaces}\label{section5} In
the present section we show the full
range of exponents $1\leq
p,q,r,s\leq\infty$ such that every
symbol in $M^{p,q}$ gives rise to a
bounded operator on $M^{r,s}$. Again,
to avoid the fact that $\cS(\R^d)$ is
not dense in $M^{r,s}(\rd)$ if
$r=\infty$ or $s=\infty$, we say that a
pseudodifferential operator $a(x,D)$ is
bounded on $M^{r,s}(\rd)$ if there
exists a constant $C>0$ such that
$\|a(x,D) f\|_{M^{r,s}} \le
C\|f\|_{M^{r,s}}$, for all $f \in
\cS(\R^d)$.\par
 We need the following auxiliary
 result.
 \begin{lemma}\label{dualita2}
Suppose that for some $1\leq p,q,
r,s\leq\infty$ the following estimate
holds:
\[
\|L_\sigma f\|_{M^{r,s}}\leq
C\|\sigma\|_{M^{p,q}}\|f\|_{M^{r,s}},\
\ \forall \sigma\in \cS(\R^{2d}),\
\forall f\in \cS(\R^d).\] Then the same
estimate is satisfied with $r,s$
replaced by $r',s'$ (even if $r=\infty$
or $s=\infty$).
\end{lemma}
\begin{proof}
The proof goes exactly as that of Lemma
\ref{dualita1}. It suffices to replace
everywhere the spaces $W(L^r,L^s)$ with
$M^{r,s}$.
\end{proof}
\begin{theorem}\label{1a}
Let $1\leq p,q,r,s\leq\infty$
such that
\begin{equation}\label{limitazionilp3}
p\leq q',\qquad
q\leq\min\{r,r',s,s'\}.
\end{equation}
Then, every symbol $a\in
M^{p,q}$ gives rise to a
bounded operator on
$M^{r,s}$, with the uniform
estimate
\[
\|a(x,D)f\|_{M^{r,s}}\lesssim
\|a\|_{M^{p,q}}\|f\|_{M^{r,s}}.
\]
\end{theorem}
\begin{proof}
The desired conclusion follows at once
by interpolation (see Lemma \ref{mo}
(iii)) from the known cases
$(p,q)=(\infty,1)$, $1\leq
r,s\leq\infty$ (see \cite[Corollary
14.5.5]{book}), and $p=q=r=s=2$ (see
\cite[Theorem 14.6.1]{book}).
\end{proof}
\begin{proposition}\label{1-2}
Suppose that, for some $1\leq
p,q,r,s\leq\infty$, $C>0$ the
estimate
\begin{equation}\label{stimamod}
\|a(x,D)f\|_{M^{r,s}}\leq C
\|a\|_{M^{p,q}}\|f\|_{M^{r,s}}\qquad
\forall a\in\cS(\R^{2d}),\
f\in\cS(\R^d)
\end{equation}
holds. Then the constraints
in \eqref{limitazionilp3}
must hold.
\end{proposition}
\begin{proof}
We first prove that
$q\leq\min\{r,r'\}$. The estimate $q\leq r'$
follows by testing
\eqref{stimamod} on the same
families of symbols and
functions as in the first
part of the proof of Proposition
\ref{t41}, taking into
account that, since the
functions $f_\lambda$
considered there have Fourier
transform supported in a
fixed compact set, it turns
out
$\|f_\lambda\|_{M^{r,s}}\asymp \|f_\lambda\|_{r}$. Precisely, using Lemma \ref{l1},
$$\|f_\lambda\|_{M^{r,s}}\asymp \|f_\lambda\|_{r}\asymp \lambda^{\frac d r-\frac d2}.$$
Moreover,
$\|a_\lambda\|_{\mpq}\lesssim\lambda^{\frac
d q-\frac d2}$ and
\[
a_\lambda(x,D)
f_\lambda(x)=\int e^{2\pi
ix\o}h(x)
h^2(\o)\,d\o,
\]
which is a non-zero Schwartz function
independent of $\lambda$, so that,
letting $\lambda\rightarrow+\infty$ in
\eqref{stimamod}, we get $q\leq r'$.
The constraint $q\leq r$ is obtained
using duality arguments, i.e. by Lemma
\ref{dualita2}.
\par Let us now prove that
$q\leq\min\{s,s'\}$. Again, we can
consider the Weyl quantization
$a(x,D)=L_\sigma$ in place of the
Kohn-Nirenberg one. Then we conjugate
the operator $L_\sigma$ with the
Fourier transform. An explicit
computation shows that $\Fur^{-1}
L_\sigma\Fur=L_{\sigma\circ\chi}$,
where $\chi(x,\o)=(\o,-x)$. On the
other hand, the map $\sigma\longmapsto
\sigma\circ\chi$ is an isomorphism of
$M^{p,q}$, so that \eqref{stimamod} is
in fact equivalent to
\[%\begin{equation}%\label{stimamod2}
\|a(x,D)f\|_{W(\Fur
L^r,L^s)}\lesssim
\|a\|_{M^{p,q}}\|f\|_{W(\Fur
L^r,L^s)}\qquad \forall
a\in\cS(\R^{2d}),\
f\in\cS(\R^d)
\]%\end{equation}
where we came back to the
Kohn-Nirenberg quantization. Then one
can test this last estimate again on
the same families of symbols and
functions as in the first part of the
proof of Proposition \ref{t41} (and in
the first part of this proof), taking
into account that, since the functions
$f_\lambda$ have Fourier transform
supported in a fixed compact set, it
turns out $\|f_\lambda\|_{W(\Fur
L^r,L^s)}\asymp \|f_\lambda\|_{s}$ (in
fact, this amounts to saying
$\|\widehat{f_\lambda}\|_{M^{r,s}}\asymp
\|\widehat{f_\lambda}\|_{\Fur L^s}$;
see Lemma \ref{lloc} (i)). Hence we get
$q\leq s'$. By duality as above then
$q\leq s$ follows.\par We finally prove
the constraint $p\leq q'$. Let
$\phi(t)=e^{-\pi|t|^2}$, $t\in\R^d$. We
test \eqref{stimamod} on the family of
symbols
\[
a'_\lambda(x,\o)=
\phi(\lambda
x)\phi(\lambda^{-1}\o),
\]
and functions
\[
f'_\lambda(x)=\phi(\lambda
x),
\]
with $\lambda\geq1$. By
\cite[Lemma 3.2]{cordero2} we
have
\[
\|a'_\lambda\|_{M^{p,q}}\asymp
\lambda^{\frac{d}{p}-\frac{d}{q'}},
\]
and
\[
\|f_\lambda\|_{M^{r,s}}\asymp\lambda^{-\frac{d}{s'}}.
\]
On the other hand, an
explicit computation gives
\[
a'_\lambda(x,D)f'_\lambda(x)=\left(a'_1(x,D)
\phi\right)(\lambda x),
\]
where $a_1(x,D) \phi$ is
still a Gaussian function.
Hence, again by \cite[Lemma
3.2]{cordero2}
\[
\|a'_\lambda(x,D)f'_\lambda\|_{M^{r,s}}
\asymp\lambda^{-\frac{d}{s'}}.
\]
Taking into account these estimates and
letting $\lambda\to+\infty$ we obtain
the desired constraint $p\leq q'$.
\end{proof}
\begin{proposition}\label{1-3}
Suppose that, for some $1\leq
p,q,r,s\leq\infty$, every symbol in
$M^{p,q}$ gives rise to a bounded
operator on $M^{r,s}$. Then the
constraints in \eqref{limitazionilp3}
must hold.
\begin{proof}
The proof is essentially the same as
that of Proposition \ref{t41bis-1},
relying on the Closed Graph Theorem and
Proposition \ref{1-2}. We leave the
details to the reader.
\end{proof}
\end{proposition}

\section{Symbols in Wiener
amalgam spaces}\label{section6} A
natural question which can be arisen is
whether the boundedness of
pseudodifferential operators on $L^2$
or, more generally, on $\mpq$, can be
attained by widening  the symbol
Sj\"ostrand class $M^{\infty,1}$ to the
Winer amalgam space $W(\cF
L^1,L^\infty)$. An example of a
function belonging to $W(\cF
L^1,L^\infty)(\rd)\setminus
M^{\infty,1}(\rd)$ is the chirp
$\f(x)=e^{\pi i |x|^2}$, $x\in\rd$, see
Theorem 14 of \cite{BGOR} and
Proposition 3.2 of \cite{cordero}. The
multiplier operator $a(x,D)f(x)= e^{\pi
i |x|^2} f(x)$ is a pseudodifferential
operator with symbol $a(x,\o)=(e^{\pi i
|\cdot|^2}\otimes 1)\phas \in W(\cF
L^1,L^\infty)(\rdd)\setminus
M^{\infty,1}(\rdd)$ and it is bounded
on $M^{p,q}$ if and only if $p=q$, see
Proposition 7.1 of \cite{fio1}. \par
The subsequent Proposition
\ref{t27bis2} shows that generally
symbols in $W(\Fur L^1,L^\infty)$ do
not produce bounded operator even on
$L^2(\rd)$. However, symbols expressed
as tensor products
$a\phas=a_1(x)a_2(\o)$, are bounded on
$M^p(\rd)$, $1\leq p\leq\infty$ (and
hence on $L^2(\rd)$), as shown in the
next result.
\begin{proposition}\label{tensoriale} If $a\phas=a_1(x)a_2(\o)$, $a_i\in W(\cF
L^1,L^\infty)$, $i=1,2$, then
$a(x,D)$ is bounded on
$M^p(\rd)$, $1\leq p\leq
\infty$.
\end{proposition}
\begin{proof} We have
\begin{equation*}
a(x,D)f(x)=
a_1(x)\cF^{-1}(a_2
\hat{f})(x)=a_1(x)[\cF^{-1}(a_2)\ast
f](x).
\end{equation*}
If $a_2\in W(\cF L^1,L^\infty)(\rd)$,
then $\cF^{-1}(a_2) \in
M^{1,\infty}(\rd)$ and
$\cF^{-1}(a_2)\ast f\in
M^{1,\infty}(\rd) \ast M^p(\rd)
\hookrightarrow M^p(\rd)$; see Lemma
\ref{mo} (iv). It remains to show that
the multiplier $U_{a_1} f(x)=a_1(x)
f(x)$ is bounded on $M^p$. This
immediately follows by the pointwise
products for Wiener amalgam spaces of
Lemma \ref{WA} (v). Indeed, $\cF
L^1\cdot \cF L^p=\cF (L^1\ast
L^p)\hookrightarrow \cF L^p$,
$L^\infty\cdot L^p \hookrightarrow
L^p$, so that, for $M^p=W(\cF
L^p,L^p)$, we have $W(\cF L^1,
L^\infty)\cdot W(\cF
L^p,L^p)\hookrightarrow W(\cF
L^p,L^p)$, as desired.
\end{proof}
\par
Observe that the previous
result does not hold if the
space $M^p$ is replaced by
$M^{p,q}$, $p\not=q$, the
counterexample being given by
the multiplier operator
$a(x,D)f(x)=e^{\pi i
|x|^2}f(x)$, which is not
bounded on $M^{p,q}$, $p\not=
q$, as discussed above.\par
We now come to a necessary
condition.
\begin{proposition}\label{t27bis}
Suppose that, for some $1\leq
p,q,r,s,r_1,r_2\leq\infty$,
$C>0$, either the estimate
\begin{equation}\label{stimacont27-2}
\|a(x,D)f\|_{W(L^{r_1},L^{r_2})}\leq
C \|a\|_{W(\Fur
L^p,L^q)}\|f\|_{W(L^{r},L^{s})},\quad
\forall a\in\cS(\R^{2d}),\
f\in\cS(\R^d),
\end{equation}
or the estimate
\begin{equation}\label{stimacont27-2-0}
\|L_\sigma
f\|_{W(L^{r_1},L^{r_2})}\leq
C \|\sigma\|_{W(\Fur
L^p,L^q)}\|f\|_{W(L^{r},L^{s})},\quad
\forall a\in\cS(\R^{2d}),\
f\in\cS(\R^d),
\end{equation}
holds. Then $q\leq r$.
\end{proposition}
\begin{proof}
We first suppose that
\eqref{stimacont27-2} holds
true. We test that estimate
on the same families of
functions and symbols as in
the proof of Proposition
\ref{t27}, but with
\eqref{h12} replaced by
\begin{equation}\label{h21}
{\rm supp}\,{h_1}\subset
B(0,1),\qquad {\rm
supp}\,{h_2}\subset B(0,1),
\end{equation}
(the other conditions being
unchanged).\par Then,
\eqref{27-2} and \eqref{27-5}
still hold, because in their
proof we did not use
\eqref{h12}. On the other
hand, we can prove that
\begin{equation}\label{27-6}
\|a_N\|_{W(\Fur
L^p,L^q)}\lesssim N^{d/q},
\end{equation}
 by using the same
 arguments in the proof of
 \eqref{27-4}, with the roles
 of $h_1,h_2$ replaced by
 $\widehat{h_2}$ and $\widehat{h_1}$,
 respectively. Precisely, we
 now choose a window $\Phi=\f\otimes\f$
 with $\f\in\cS(\R^d)$ supported in
 $B(0,1)$. Then the functions
 $V_\Phi
 b_n(z_1,z_2,\omega_1,\omega_2)$
 vanish unless $z_1\in
 B(-n,2)$, $z_2\in B(0,2)$, so
 that they have disjoint
 supports, as well as
 $\|V_\Phi
 b_n(z_1,z_2,\cdot,\cdot)\|_p$.
 Hence one deduces that
 \begin{align*}
 \|a_n\|_{W(\Fur
 L^p,L^q)}&\asymp \left(\int_{\R^{2d}}
\|\sum_{n\in\Lambda_N}V_\Phi
 b_n(z_1,z_2,\cdot,\cdot)\|_{p}^q\,d
z_1\,dz_2\right)^{1/q}\\
&=
 \left(\sum_{n\in\Lambda_N}
\int_{B(-n,2)}\int_{B(0,2)}
\|V_\Phi
 b_n(z_1,z_2,\cdot,\cdot)\|_{p}^q\,
 dz_1\,dz_2 \right)^{1/q}.
\end{align*}
On the other hand, by Young's
inequality, we have
\[
\|V_{\Phi}{b_n}(z_1,z_2,\cdot,\cdot)\|_{p}\leq
\|\widehat{b_n}\|_p,
\]
which is independent of $n$. Hence
\eqref{27-6} follows.\par We now
suppose that the estimate
\eqref{stimacont27-2-0} holds. Since
the arguments are similar to those just
used, we only sketch the main point of
the proof. We test the estimate
\eqref{stimacont27-2-0} on the family
of functions $f_N(x)=h(Nx)$, where $h$
is a smooth real-valued function, with
$h\geq0$, $h(0)=1$, supported in a ball
$B(0,\epsilon)$, for a sufficiently
small $\epsilon$, and symbols
\[
\sigma_N(x,\omega)=\sum_{n\in\Lambda_N}b_n(x,\o),\
{\rm where}\ b_n(x,\o)=(M_{-2n}h_1)(x)
(T_n h_2)(\o).
\]
The lattice $\Lambda_N$ is defined in
\eqref{reticolo0} and $h_1,h_2$ are two
Schwartz functions in $\R^d$ such that
$h_1$ and $\widehat{h}_2$ are real
valued, with $h_1(0)=1$,
$\widehat{h_2}(0)=1$, and satisfying
\eqref{h21}. By using the definition
\eqref{equiv1} we see that $L_\sigma$
has integral kernel
\[
K(x,y)=\Fur_2^{-1}\sigma\left(\frac{x+y}{2},x-y\right)=
\sum_{n\in\Lambda_N} e^{-4\pi in y}
h_1\left(\frac{x+y}{2}\right)\widehat{h_2}(y-x).
\]
Hence, by arguing as in the
proof of \eqref{27-1} one
obtains, for a suitable
$\delta>0$,
\[
|L_{\sigma_N} f_N(x)|\gtrsim
1,\ {\rm for}\ x\in
B(0,\delta),
\]
which implies
\[
\|L_{\sigma_N}
f_N\|_{W(L^{r_1},L^{r_2})}\gtrsim
1.
\]
The arguments in the
first part of the present
proof,  with essentially no changes, show that
\[
\|\sigma_N\|_{W(\Fur
L^p,L^q)}\lesssim N^{d/q}.
\]
Combining these estimates
with
$\|f_N\|_{W(L^r,L^s)}\lesssim
N^{-d/r}$ and letting
$N\to+\infty$ give the
desired conclusion.
\end{proof}
\begin{proposition}\label{t27bis2}
Suppose that, for some $1\leq
p,q,r,s\leq\infty$, every
symbol $a\in W(\Fur L^p,L^q)$
gives rise to a bounded
operator on $W(L^r,L^s)$.
Then $q\leq r$. The same
happens if one replaces the
Kohn-Nirenberg operator
$a(x,D)$ by the Weyl one
$L_a$.
\end{proposition}
\begin{proof}
The result follows from
Proposition \ref{t27bis} by
arguing as in the proof of
Proposition \ref{t41bis-1}.
\end{proof}

As anticipated in the Introduction, we
address now to the problem of the
invariance of the Wiener amalgam spaces
$W(\Fur L^p,L^q)$ under the action of
the operator $\mathcal{U}$ in Theorem
\ref{teo-scambio}, which expresses the
Kohn-Nirenberg symbol of an operator in
terms of the Weyl one. The lack of
invariance, expressed by the following
result, justifies the fact that the
necessary conditions in this section
were proved for both Kohn-Nirenberg and
Weyl quantizations.
\begin{proposition}\label{6punto4}
Let $1\leq p,q\leq\infty$. If $p\not=q$, the operator $\mathcal{U}$ in
Theorem \ref{teo-scambio}
does not map $W(\Fur L^p,L^q)$ into itself.
\end{proposition}
\begin{proof}
Consider the symmetric matrix
$
B=\frac{1}{2}\left[\begin{matrix}
0&I_d\\
I_d&0\end{matrix}\right], $ and the
symplectic matrix
$\mathcal{A}=\left[\begin{matrix}
I_{2d}&B\\
0&I_{2d}\end{matrix}\right]$. It
follows from Theorem 4.51 of
\cite{folland} that the operator
$\mathcal{U}$ is exactly the
metaplectic operator associated with
the matrix $\mathcal{A}$. We write
$\mathcal{U}=\mu(\mathcal{A})$. Hence,
as a consequence of Theorem 4.1 of
\cite{cordero2}, we deduce that
$\mathcal{U}$ and
$\mathcal{U}^{-1}=\mu(\mathcal{A}^{-1})$
map $W(\Fur L^q,L^p)$ into $W(\Fur
L^p,L^q)$ continuously
 for every $1\leq p\leq
 q\leq\infty$.
\par
Now, assume $p<q$. Let $f$ be any
distribution in $W(\Fur L^q,L^p)$;
therefore, by the boundedness result we have
just recalled, $\mathcal{U}^{-1}f\in
W(\Fur L^p,L^q)$. Suppose, by
contradiction, that $\mathcal{U}$ maps
$W(\cF L^p,L^q)$ into itself. Then one
would obtain
$f=\mathcal{U}\mathcal{U}^{-1} f\in
W(\Fur L^p,L^q)$, and therefore the
inclusion $W(\Fur L^q,L^p)\subseteq
W(\Fur L^p,L^q)$, which is false.\par
Suppose now $p>q$. Assume, by
contradiction, that for every $f\in
W(\Fur L^p, L^q)$ it turns out that
$\mathcal{U}f\in W(\Fur L^p, L^q)$.
Then we would have
$f=\mathcal{U}^{-1}\mathcal{U}f\in
W(\Fur L^q, L^p)$, and therefore the
inclusion $W(\Fur L^p,L^q)\subseteq
W(\Fur L^q,L^p)$, which is false.
\end{proof}
\begin{remark}\rm
A concrete example of a Weyl
symbol $\sigma\in W(\Fur
L^1,L^\infty)$ such that the
corresponding Kohn-Nirenberg
symbol $a=\mathcal{U}\sigma$
does not belong to $W(\Fur
L^1,L^\infty)$ is provided by
$\sigma=\mathcal{U}^{-1}\delta$;
therefore $a=\delta$. To see
this, observe, first of all,
that the Dirac distribution
$\delta$ belongs to $W(\Fur
L^\infty,L^1)$. Indeed, for a
given window
$g\in\mathcal{S}(\R^d)\setminus\{0\}$,
\[
\|\delta\|_{W(\Fur
L^\infty,L^1)}=\int\|\widehat{\delta
T_x \overline{g}}\|_\infty
dx=\int|g(-x)|dx<\infty.
\]
Hence, by
Theorem 4.1 of
\cite{cordero2}, we infer $\sigma\in
W(\Fur L^1,L^\infty)$. On the
other hand, it is clear that
$\delta$ does not belong to
$W(\Fur L^1,L^\infty)$ (which
consists only of continuous
functions).
\end{remark}


\begin{thebibliography}{10}

\bibitem{BGOR}
\'A. B\'enyi, K. Gr\"ochenig,
K. Okoudjou and L.G. Rogers,
{Unimodular Fourier
multipliers for modulation
spaces}, J. Funct. Anal. 246
(2007), 366-384.

\bibitem{CM78}
R.R. Coifman, Y. Meyer.
\newblock Au del\`a des op\'erateurs pseudo-diff\'erentiels. {\em Ast\'erisque},
57:1--185, 1978.

\bibitem{Co75}
H.O. Cordes.
\newblock On compactness of commutators of multiplications and
              convolutions, and boundedness of pseudodifferential operators. {\em J. Funct. Anal.},
18:115--131, 1975.
\bibitem{calderon71}
A.P. Calder\'on, R. Vaillancourt. On
the boundedness of pseudo-differential
operators. {\em J. Math. Soc. Japan},
23:374--378, 1971.
\bibitem{calderon72}
A.P. Calder\'on, R. Vaillancourt. A
class of bounded pseudo-differential
operators. {\em Proc. Nat. Acad. Sci.},
U.S.A. 69:1185--1187, 1972.

\bibitem{cordero}
E.~Cordero and F. Nicola.
\newblock {S}trichartz
estimates in Wiener amalgam spaces for
the Schr\"odinger equation.
\newblock {\em Math. Nachr.}, 281(1):25--41, 2008.

\bibitem{cordero2}
E.~Cordero and F. Nicola.
\newblock {Metaplectic representation
on  Wiener amalgam spaces and applications
to the Schr\"odinger equation}.
\newblock {\em J. Funct.
Anal.}, 254:506--534, 2008.

\bibitem{CNloc09}
E.~Cordero and F. Nicola.
\newblock Sharp Continuity Results for the Short-Time Fourier Transform and for Localization Operators.
{\em Preprint}, 2009. Available at ArXiv:0904.1508v1.


\bibitem{fio1}
E.~Cordero, F. Nicola and L. Rodino.
\newblock {T}ime-frequency  {A}nalysis of {F}ourier {I}ntegral {O}perators.
\textit{Comm. on Pure and Appl. Anal.},  to appear. Available at ArXiv:0710.3652v1.

\bibitem{feichtinger80} H.~G.
Feichtinger.
\newblock Banach convolution algebras of {W}iener's type,
\newblock In {\em Proc. Conf. ``Function, Series, Operators",
 Budapest August 1980}, Colloq. Math. Soc. J\'anos Bolyai, 35,  509--524, North-Holland, Amsterdam,
 1983.

\bibitem{F1}  H.~G.~Feichtinger.
\newblock Modulation spaces on locally
compact abelian groups,
\newblock {\em Technical Report, University Vienna, 1983.} and also in
\newblock {\em Wavelets and Their Applications},
M. Krishna, R. Radha,  S. Thangavelu,
editors,
\newblock Allied Publishers, 99--140, 2003.

\bibitem{feichtinger83}
H.~G. Feichtinger.
\newblock Banach spaces of distributions of {W}iener's type and interpolation.
\newblock In {\em Proc. Conf. Functional Analysis and Approximation,
 Oberwolfach August 1980},  Internat. Ser. Numer. Math., 69:153--165. Birkh\"auser, Boston, 1981.
\bibitem{hans85}
H.~G. Feichtinger.
\newblock {B}anach {S}paces of {D}istributions {D}efined by  {D}ecomposition {M}ethods, I.
\newblock {\em Math. Nachr.}, 123:97--120, 1985.

\bibitem{fe89-1}
H.~G. Feichtinger.
\newblock {A}tomic characterizations of modulation spaces through {G}abor-type
representations.
\newblock In  {\em {P}roc. {C}onf. {C}onstructive {F}unction {T}heory},  {\em {R}ocky {M}ountain {J}. {M}ath.}, 19:113--126, 1989.

\bibitem{folland}
G.~B. Folland.
\newblock {\em Harmonic Analysis in Phase Space}.
\newblock Princeton Univ. Press, Princeton, NJ, 1989.

\bibitem{fournier-stewart85}
J.~J.~F. Fournier and J.~Stewart.
\newblock Amalgams of ${L}\sp p$ and $l\sp q$.
\newblock {\em Bull. Amer. Math. Soc. (N.S.)}, 13(1):1--21, 1985.

\bibitem{book}
K.~Gr{\"o}chenig.
\newblock {\em Foundations of Time-Frequency Analysis}.
\newblock Birkh\"auser, Boston, 2001.

\bibitem{grochenig-ibero}
K. Gr\"ochenig. Time-frequency analysis
of Sjöstrand's class. {\it Rev. Mat.
Iberoam.,} 22:703--724, 2006.

\bibitem{GH99}
K.~Gr{\"o}chenig and C. Heil.
\newblock {M}odulation spaces and pseudodifferential operators.
\newblock {\em Integral Equations Operator Theory}, 34:439--457, 1999.

\bibitem{GH04}
K.~Gr{\"o}chenig and C. Heil.
Counterexamples for boundedness of
pseudodifferential operators. {\it
Osaka J. Math.}, 41:681--691, 2004.

\bibitem{Heil03}
C. Heil.
\newblock {A}n introduction to weighted {W}iener amalgams.
\newblock In M. Krishna, R. Radha and S. Thangavelu, editors, {\em
Wavelets and their Applications},
183--216. Allied Publishers Private
Limited, 2003.

\bibitem{hormander65}
 L. H\"ormander. Pseudo-differential operators. {\em
  Comm. Pure Appl. Math}, 18:501--517, 1965.
\bibitem{hormander-book}
L. H\"ormander. {\it The analysis of
linear partial differential operators.
III. Pseudo-differential operators.}
Grundlehren der Mathematischen
Wissenschaften, 274. Springer-Verlag,
Berlin, 1994.

\bibitem{kohn-nirenberg65}
J.J. Kohn and L. Nirenberg. An algebra
of pseudo-differential operators. {\em
Comm. Pure Appl. Math}, 18:269--305,
1965.
\bibitem{ka76}
T. Kato. Boundedness  of some
pseudo-differential operators. {\it
Osaka J. Math.} 13:1--9, 1976.


\bibitem{labate2}
D. Labate. Pseudodifferential operators
on modulation spaces. {\em J. Math.
Anal. Appl.,} 262:242--255, 2001.

\bibitem{Na77}
M. Nagase. The $L^p$-boundedness of pseudo-differential operators with non-regular symbols. {\it
Comm. Partial Differential Equations} 2:1045--1061, 1977.

\bibitem{kasso07} K.A. Okoudjou.
\newblock A Beurling-Helson type theorem for
 modulation spaces. {\it J. Funct. Spaces Appl.},
 to appear.
 Available at http://www.math.umd.edu/$\sim$kasso/publications.html.

\bibitem{sjostrand94}
J. Sj\"ostrand. An algebra of
pseudodifferential operators. {\em
Math. Res. Lett.}, 1:185--192, 1994.

\bibitem{sjostrand95}
J. Sj\"ostrand. Wiener type algebras of
pseudodifferential operators. {\em
S\'eminaire sur les \'Equations aux
D\'eriv\'ees Partielles}, 1994--1995,
Exp. No. IV, 21 pp., \'Ecole Polytech.,
Palaiseau, 1995.

\bibitem{stein93}
E. M. Stein,
\newblock \emph{ Harmonic analysis}.
\newblock Princeton University Press, Princeton, 1993.

\bibitem{Su88}
M. Sugimoto. $L^p$-boundedness of
pseudo-differential operators
satisfying Besov estimates II. {\em  J.
Math. Soc. Japan}, 40:105--122, 1988.

\bibitem{sugimoto2}
M. Sugimoto. Pseudo-differential
operators on Besov spaces. {\em Tsukuba
J. Math.}, 12:43--63, 1988.

\bibitem{tao}
T. Tao.
\newblock {\it Nonlinear Dispersive Equations: Local and Global Analysis}.
\newblock CBMS Regional Conference Series in Mathematics, Amer. Math. Soc., 2006.

\bibitem{taylor1}
M.E. Taylor. {\em Partial differential
equations, II, III}. Applied
Mathematical Sciences, 116-117.
Springer-Verlag, New York, 1996-1997.
\bibitem{taylor2}
M.E. Taylor. {\em Tools for PDE.
Pseudodifferential operators,
paradifferential operators, and layer
potentials}. Mathematical Surveys and
Monographs, 81. Amer. Math. Soc.,
Providence, RI, 2000.

\bibitem{Toft04} J.~Toft.
\newblock Continuity properties for modulation spaces, with applications to
  pseudo-differential calculus. {I}.
\newblock {\em J. Funct. Anal.}, 207(2):399--429, 2004.

\bibitem{Toftweight}
J.~Toft.
\newblock Continuity properties for modulation spaces, with applications
              to pseudo-differential calculus. {II}.
\newblock {\em Ann. Global Anal. Geom.}, 26(1):73--106, 2004.
\end{thebibliography}
\end{document}